\documentclass[10pt,reqno]{amsart}
\usepackage[margin=1in]{geometry}  
\usepackage[english]{babel}
\usepackage[utf8]{inputenc}
\usepackage{amsmath,amssymb,amsthm, comment}
\usepackage{amsaddr}
\usepackage{hyperref}
\usepackage{cases}

\usepackage{dgmacros}
\usepackage{titlecaps}
\title{A Priori Estimates for a Relativistic Liquid With Free Surface Boundary}

\author{Dan Ginsberg}

\date{\today}

\begin{document}
\begin{abstract}
  We prove energy estimates
  for a relativistic free liquid body
  with sufficiently small fluid velocity in a general Einstein
  spacetime. These estimates control Sobolev norms of the fluid velocity
  and enthalpy in the interior as well as Sobolev norms of the second
  fundamental form on the boundary.
\end{abstract}
\maketitle
\section{The Relativistic Euler Equations}

For a given globally hyperbolic spacetime $(M, g)$
we consider the motion of an isentropic relativistic perfect fluid:
\begin{align}
  &\nabla_u u^\mu + \frac{1}{\epsilon + p} \aPi^{\mu\nu}\nabla_\nu p = 0,
  &&\textrm{ in } \D \label{int:consu}\\
  &\nabla_u \rho + \rho \nabla_\mu u^\mu = 0, &&\textrm{ in } \D,
  \label{int:consm}
\end{align}
for $\mu,\nu = 0,.., 3$. We are employing Einstein notation and summing
over all repeated upper and lower indices.
Here, $u$ is the four-velocity of the fluid, which is a unit
timelike future-directed vector field. The variables $\epsilon, p, \rho$ are the energy density,
pressure, and mass density, respectively, and are all assumed to be
non-negative in $\D$, and $\Pi$ denotes the projection to the orthogonal
complement of $u$:
\begin{equation}
 \Pi^{\mu\nu} = g^{\mu\nu} + u^\mu u^\nu.
 \label{int:proj}
\end{equation}
We are writing
$\nabla$ for the covariant derivative associated to $g$ and
$\nabla_u  = u^\mu\nabla_\mu$ and are working in units so that the speed of
light is 1. The set
$\D \subset M$
is to be determined, and will satisfy the following constraints.

We suppose that $M$ is foliated by spacelike hypersurfaces $M_t$.
Let $\D_0$ be a bounded domain diffeomorphic to the unit
ball and contained in $M_0$.
Given a velocity field
$(u_0)^\mu$, energy density $\epsilon_0$, mass density $\rho_0$ and
pressure $p_0$ on $\D_0$ satisfying:
\begin{align}
    & (u_0)^\mu (u_0)_\mu = -1,&&\textrm{ in } \D_0,\\
    &\epsilon_0, \rho_0, p_0 > 0 &&\textrm{ in } \D_0,\\
    &\epsilon_0, \rho_0 > 0 &&\textrm{ on } \pa \D_0,\label{liquideps}\\
    &p_0 = 0 &&\textrm{ on } \pa \D_0 \label{liquid},
\end{align}
we want the domain
 $\D = \bigcup\limits_{0 \leq t \leq T} \{t\} \times \D_t$ to satisfy:
\begin{align}
  &u^\mu u_\mu = -1 &&\textrm{ in } \D,\label{int:norm}\\
  &\epsilon, \rho, p > 0 &&\textrm{ in } \D,\label{int:liquid2}\\
  &p = 0 &&\textrm{ on } \Lambda, \label{int:vacuum}\\
  &u^\mu N_\mu = 0 &&\textrm{ on } \Lambda,
  \label{int:free}
\end{align}
where $N_\mu$ is the unit conormal to $\Lambda = \bigcup\limits_{0 \leq t \leq T}
\{t\} \times \pa \D_t$.

In the non-relativistic setting, it is well known
(see \cite{Ebin1987}) that the analog of the above problem
is ill-posed unless the
``Taylor sign condition'' holds. This is
the requirement that:
\begin{align}
  \nabla_N p > -\delta > 0 \textrm{ on } \Lambda,
  && \textrm{ where } \nabla_N = N^i \nabla_i.
  \label{int:tsc}
\end{align}

The above equations do not close, because there are six independent
variables but only four equations. For an isentropic fluid,
the energy density $\epsilon$ is a function of the mass
density $\rho$,
$\epsilon = \epsilon(\rho)$
for some smooth, non-negative, increasing function $\epsilon$,
which is determined from the other thermodynamical variables
$\rho, p$ by solving:
\begin{equation}
  \frac{d \epsilon}{d \rho} = \frac{\epsilon + p}{\rho}.
  \label{int:therm}
\end{equation}

We will also assume that the fluid is barotropic, which means that
$p = p(\epsilon)$ for a given smooth, non-negative increasing function $p$.
In this case, \eqref{int:vacuum} implies that $\epsilon \equiv \epsilon_0$
on the boundary and therefore $\rho = \rho_0$ on $\pa \D_t$ as well. The
assumption that $\rho_0, \epsilon_0 > 0$ corresponds to an equation
of state modeling a liquid.

Using \eqref{int:therm}, given any one of the variables $p,\rho, \epsilon$, we can
determine the other two. In fact it is convenient to work in terms of
the (specific) enthalpy $\sqrt{\sigma}$:
\begin{equation}
  \sqrt{\sigma} = \frac{\epsilon + p}{\rho}.
  \label{int:enth}
\end{equation}
If $p = p(\epsilon)$ is given,
then both $\epsilon, \rho$ can be determined from \eqref{int:therm}
once $\sigma$ is known.
From now on, we will think of $\sigma$ as the fundamental thermodynamical
variable, so $(p, \rho, \epsilon) = (p(\sigma), \rho(\sigma), \epsilon(\sigma))$.

If we set
\begin{equation}
  V^\mu = \sqrt{\sigma} u^\mu,
  \label{int:vdef}
\end{equation}
then using \eqref{int:therm}, the equations
\eqref{int:consu} and \eqref{int:consm} can be written in terms of
$V$ and $\sigma$:
\begin{align}
  &\nabla_V V^\mu + \frac{1}{2} g^{\mu\nu}
  \nabla_\nu \sigma = 0,\label{int:dtv}\\
  &\nabla_V e(\sigma) + \nabla_\mu V^\mu = 0,
  \label{int:dtesig}
\end{align}
where $e(\sigma) = \log (\rho(\sigma)/\sqrt{\sigma})$, and the boundary
conditions are:
\begin{align}
  \sigma &= \sigma_0, && \textrm{ on } \pa \D_t,\label{int:liquidsig}\\
  N_\mu V^\mu &= 0, && \textrm{ on } \pa \D_t,
 \label{int:free2}
\end{align}
with $\sigma_0 = \frac{\epsilon_0 + p_0}{\rho_0}$.
For a derivation of \eqref{int:dtv}-\eqref{int:dtesig}, see, for example
\cite{Christodoulou1995a}. We will derive energy estimates for
the system \eqref{int:dtv}-\eqref{int:free2} with
sufficiently small fluid velocity, assuming that
\eqref{int:tsc} holds.

\subsection{Previous results}
In the non-relativistic setting, if $\D_t$ is unbounded and the
fluid is incompressible and irrotational, the above is known as the ``water waves'' problem.
In this case,
global existence for small initial data has been proven using a variety of
techniques;
see for example \cite{W99}, \cite{Germain2009} as well as
the recent survey article \cite{Ionescu2017} for a detailed history of that problem.

For an incompressible bounded fluid body with nonvanishing curl,
in \cite{CL00}, Christodoulou-Lindblad prove energy estimates that
control the fluid velocity and the second fundamental form of the boundary,
and our approach is modeled on theirs.
Local well-posedness for this case was first proved
by Lindblad in \cite{L05a} using a Nash-Moser iteration.
Coutand-Shkoller \cite{CS07} were able to avoid the use of
Nash-Moser by using a tangential smoothing operator and fractional Sobolev
spaces.
In his PhD. thesis, Nordgren
\cite{Nordgren2008} modified the methods of \cite{CS07}
and proved local well-posedness for a self-gravitating
incompressible fluid.

In the case of a compressible liquid body,
local well-posedness was again proved by Lindblad in \cite{Lindblad2005}
using a Nash-Moser iteration. In \cite{CHS13}, the authors showed
local well-posedness by introducing an ``artificial viscosity'' term
and studying the zero viscosity limit.
See also the recent paper \cite{Lindblad2017} where the estimates
from \cite{CL00} are extended to handle a compressible liquid.
The case that $\rho = 0$ on $\pa \D_t$ is known as a gas;
see for example \cite{CLS10},  for a priori estimates
and \cite{Jang2015}, \cite{Makino2017}, \cite{CHS13}
for local well-posedness. See also \cite{Hadzic2017},
where a nontrivial family of global-in-time solutions are constructed for
a self-gravitating compressible gas.

Moving now to the free boundary problem for the relativistic Euler equations,
the monographs \cite{Christodoulou2007}, \cite{Christodoulou2017}
and \cite{Holzegel2016} contain a detailed analysis
of shock formation in the exterior of an irrotational fluid body.
For the interior of a liquid body with non-vanishing vorticity,
Oliynyk \cite{Oliynyk2017a} recently obtained energy estimates
in in Lagrangian coordinates. The approach and estimates there
are quite different from the present paper;
Oliynyk estimates Lagrangian derivatives of the fluid velocity and the
Lagrangian coordinate map,
while we derive estimates for Eulerian derivatives of the velocity
and the second fundamental form.
We remark that the assumption \eqref{assump:stat}
is not needed for the estimates in \cite{Oliynyk2017a}; see
the discussion after \eqref{assump:stat}.
Oliynyk also recently announced a proof
of local well-posedness for this problem; see \cite{Oliynyk2017}.
See also \cite{HSS15} and  \cite{Jang2016} for energy estimates
for relativistic gas bodies.

\subsection{Energy estimates}
We begin by recalling the basic energy associated to the system
\eqref{int:consu}-\eqref{int:consm} with boundary conditions
\eqref{int:vacuum}-\eqref{int:free}.
The equations \eqref{int:consu} are a consequence of the Einstein equations:
\begin{equation}
  R^{\mu\nu} - \frac{1}{2} R g^{\mu\nu} = T^{\mu\nu},
  \label{int:emtensor}
\end{equation}
where $T^{\mu\nu}$ is the energy-momentum tensor for a perfect fluid, defined
by:
\begin{equation}
  T^{\mu\nu} = (\epsilon + p) u^\mu u^\nu + p g^{\mu\nu},
  \label{int:em}
\end{equation}
and $R^{\mu\nu}, R$ are the Ricci and scalar curvatures of $g$, respectively.
By the second Bianchi identity, $\nabla_\mu T^{\mu\nu} = 0$, and
\eqref{int:consu} is then obtained by considering the component of
\eqref{int:emtensor} orthogonal to $u$ (see \cite{Wald1984} or
\cite{Rezzolla2013} for a derivation).

We let $\tau_\mu$ denote the future-directed unit normal to the time slices
$\D_t$ and we let $\sg$ denote the induced metric on $\D_t$ (see section
 \ref{sec:spatsec}).
Integrating $\tau_\nu \nabla_\mu T^{\mu\nu} = 0$ over $\D$ gives:
\begin{multline}
 0 = \int_{\D} \tau_\nu \nabla_\mu T^{\mu\nu} \dvol\\
 = \int_{\D_t}(-\tau_\nu \tau_\mu) T^{\mu\nu}\dvols
 -\int_{\D_0} (-\tau_\nu \tau_\mu)T^{\mu\nu}\dvols
 + \int_{\Lambda} \tau_\nu N_\mu T^{\mu\nu} \dvolbdy
 - \int_{\D} \nabla_\mu(\tau_\nu) T^{\mu\nu}\dvol.
 \label{en:ibparg}
\end{multline}
Here, $\dvols$ and $\dvolbdy$ denote the volume element
on $\D_t$ and $\Lambda$, respectively.
The boundary term is:
\begin{equation}
 \int_{\Lambda} T^{\mu\nu} N_\mu \tau_\nu \dvolbdy
 = \int_{\Lambda}\bigg( (p + \rho) (u^\mu N_\mu) (u^\nu\tau_\nu)
 + p N^\mu \tau_\mu\bigg) \dvolbdy = 0,
 \label{bc}
\end{equation}
by \eqref{int:vacuum} and \eqref{int:free}.

Writing $u^\tau = u^\mu \tau_\mu$, we have:
\begin{equation}
  \int_{\D_t} (\epsilon + p)\big(u^\tau\big)^2
  + p \,(\tau^\mu \tau_\mu) \dvols =
  \int_{\D_t} \epsilon \big(\utau\big)^2 + p\big( (\utau)^2 -1\big) \dvols.
  \label{}
\end{equation}
Because $u$ is a unit timelike vector, $1 = -u^\mu u_\mu =
(\utau)^2 - \sg(u, u)$,
so if we define:
\begin{equation}
 \E_0(t) = \int_{\D_t} \epsilon (\utau)^2 + p \sg(u,u)\dvols,
 \label{e0def}
\end{equation}
then \eqref{en:ibparg} and \eqref{bc} give:
\begin{equation}
 \E_0(t) = \E_0(0) + \int_{\D} T^{\mu\nu} (\nabla_\mu \tau_\nu)\dvol.
 \label{e0bd}
\end{equation}
We can write $T^{\mu\nu} \nabla_\mu \tau_\nu = \frac{1}{2} T^{\mu\nu}
(\nabla_\mu\tau_\nu + \nabla_\nu \tau_\mu) = \frac{1}{2}T^{\mu\nu}
\L_\tau g_{\mu\nu}$,
where $\L$ denotes the Lie derivative. In the special case that $\tau$
is a Killing field (IE that $(M, g)$ is stationary),
the second term in \eqref{e0bd} vanishes and the energy \eqref{e0def} is conserved.

We would like a higher-order version of the energy \eqref{e0def}.
Following \cite{CL00} and \cite{Lindblad2016},
the energies that we construct
for the equations \eqref{int:dtv} have interior terms and a boundary
term:
\begin{multline}
  \E^{k,\ell}(t)
  = \frac{1}{2}\int_{\D_t}
  \bigg(\big(\tau_\mu\tau_\nu + \sg_{\mu\nu}\big)
  Q(\nao^k \nabla_u^\ell V^\mu, \nao^k \nabla_u^\ell V^\nu)
  \bigg)
  \frac{\sqrt{\sigma}}{|\ubar| - \utau} \dvols\\
  + \frac{1}{4} \int_{\D_t}
  Q(\nao^k \nabla_u^\ell
  \sigma, \nao^k\nabla_u^\ell\sigma)\,\,(-\Vtau) e'(\sigma) \dvols\\
  +  \frac{1}{4} \int_{\pa \D_t}
  Q(\nao^{k}\nabla_u^\ell
  \sigma, \nao^{k}\nabla_u^\ell\sigma)
  \frac{(-\Vtau)}{|\nabla_N \sigma|}\dvolbdy,
  \label{int:ekl}
\end{multline}
Here, $\ubar$ is the spatial part of $u$ and $|\ubar| = \sg(\ubar, \ubar)$.
See section \ref{sec:defsec} for notation.
Recall that $\utau < 0$ so the $\E^{k,\ell}$
are non-negative.
Here, $\sd$ is the
covariant derivative on $\D_t$, and
$Q$ is a positive definite quadratic form which is the
inner product of the tangential components on the boundary and increases
to the full norm in the interior, which we define as follows.
We let $\widetilde{\n}$ denote the outer unit conormal to $\pa \D_t$,
and we let $\n$ denote an extension of $\widetilde{\n}$ to a neighborhood
of the boundary. Writing $d = d(x) = \dist(x, \pa\D_t)$ for the geodesic
distance to the boundary, we let $\chi = \chi(d)$ denote a smooth
function which is zero in the interior of $\D_t$ and one near $\pa \D_t$.
We then define the projection:
\begin{align}
  \Pis_i^j = \delta_i^j - \chi(d)\n_i \n^j,
  \label{Pisdef}
\end{align}
where $\n^j = \bar{g}^{j\ell}\n_\ell$ and Latin indices $i,j,k,...$ run
over $1,2,3$.
For $x \in \pa \D_t$, this is the projection from $T_x(\D_t)$ onto
$T_x(\pa \D_t)$, and in the interior it is the identity map. We will also write:
\begin{equation}
 \gamma^{ij} = \sg^{ik}\Pis^j_k
 \label{}
\end{equation}
for the metric on $\pa \D_t$ extended to a neighborhood of the boundary.
If
$\alpha, \beta$ are $(1, r)$ tensors on $\D$ then we define:
\begin{equation}
  Q(\alpha^\mu, \beta^\nu) = \gamma^{i_1 j_1} \cdots \gamma^{i_r j_r}
  \alpha_{i_1\cdots i_r}^\mu \beta_{j_1\cdots j_r}^\nu.
  \label{}
\end{equation}

For each $\mu,\nu$,
$Q(\alpha^\mu, \beta^\nu)$ is the inner product
of the spatial components of $\alpha$ and $\beta$ which are tangent
to $\pa \D_t$, and in the interior of $\D_t$ it is the inner product of
spatial components of $\alpha$ and $\beta$.

To control the curl we define:
\begin{equation}
 \K^{r}(t) = \int_{\D_t} | \curl \sd^{r-1}V|^2 \dvols.
 \label{}
\end{equation}

To control the enthalpy $\sigma$, we
use the fact that $\sigma$ satisfies the following wave equation:
\begin{equation}
  -\nabla_V^2 e(\sigma) + \frac{1}{2} \nabla_\mu \nabla^\mu \sigma
  = -(\nabla_\mu V^\nu)(\nabla_\nu V^\mu),
  \label{int:wave}
\end{equation}
which follows by taking the divergence of \eqref{int:dtv} and subtracting it
from $\nabla_V$ applied to \eqref{int:dtesig}.
Using \eqref{int:wave}, we are able to prove energy estimates for the
quantities:
\begin{multline}
  \EW^{r}(t) = \frac{1}{2} \int_{\D_t}
  \bigg( (\nabla_u^{r+1} \sigma)^2
  + \aPi^{\mu\nu} \big(\nabla_\mu \nabla_u^r \sigma\big)
  \big( \nabla_\nu \nabla_u^r \sigma\big)
   \frac{1}{|\bar{u}|-\utau}\bigg) \dvols  +\\
  \frac{1}{2} \int_{\D_t}   (\nabla_u^{r+1} \sigma)^2
  (\eta^{-2} - 1)
  (-\utau) \dvols.
  \label{int:waveen}
\end{multline}
Recall that $\aPi$ is the projection onto the local simultaneuous
space of the fluid, defined in \eqref{int:proj}. Here, $\eta$ is the
sound speed, $\eta(\epsilon) \equiv \sqrt{p'(\epsilon)}$. Since the
speed of light is one in these units, we require
$\eta \leq 1$, which makes the last term non-negative. We remark that for an
equation of state with $\eta = 1$, the operator on the right-hand side
of \eqref{int:wave} reduces to the wave operator in the metric $g$
and \eqref{int:waveen} reduces to the usual energy for the wave equation on $(M, g)$,
ignoring the factor of $( |\ubar| - \utau)^{-1}$.

The energies we consider are then:
\begin{equation}
 \E^r(t) = \sum_{k + \ell = r}\bigg( \E^{k,\ell}(t)
 + \K^{k,\ell}\bigg) + \EW^{r}(t).
 \label{}
\end{equation}
In order to close our energy estimates, we will need to assume that the
spatial part of the fluid velocity is bounded. We write:
\begin{equation}
 \lambda = \frac{|\ubar|}{|\utau|}.
 \label{}
\end{equation}
Note that $\lambda < 1$ because $u$ is a unit timelike
vector field.  We assume that:
\begin{equation}
 \lambda \leq \lambda^*,
 \label{int:stat}
\end{equation}
where $\lambda^* = \lambda^*(r)$ is a sufficiently small constant.
See the discussion after equation \eqref{matderivident}.

We also need some a priori assumptions:
\begin{align}
 &|\nabla u| + |\nabla \sigma| \leq M_1, \textrm{ in } \D_t,\\
 &|\nabla^2\sigma| \leq M_2 \textrm{ on } \pa \D_t.
 \label{int:ap}
\end{align}
In addition, we require that the equation of state $p = p(\epsilon)$ is
reasonably well-behaved; see section \ref{fluidassump} for the specific
assumptions.

Our main theorem is then:
\begin{theorem}
  \label{int:main}
  Suppose that the above assumptions hold.
  For $r \geq 5$ there is a
  $T^*$ so that for $0 \leq t \leq T^*$:
    \begin{equation}
      \E^r(t) \leq 2 \E^r(0).
    \label{}
  \end{equation}
\end{theorem}
See theorem \ref{mainthm} for a precise version of this statement.

The quantities \eqref{int:ekl}
do not directly control all space derivatives of $V$ and
$\sigma$. To get back control over all derivatives of $V$, we use elliptic
estimates as in
\cite{CL00},
which require that we control $\sdiv V$ and $\scurl V$,  the divergence and
curl of $V$ taken over only spatial components; see section \ref{sec:spatsec}.
To control
$\sdiv V$, we will use equation \eqref{int:dtesig} and to control $\scurl V$, we will
use equation \eqref{int:dtv}. In each case, we will need to be able to control
quantities involving $\nabla_\tau V$, and we control
this if we control $1/\utau$ times $\nabla_u V$ and $\ubar^\mu\sd_\mu V$. We can then
use the equation \eqref{int:dtv} and the fact that $\lambda$ is small
to handle these terms. We will also use these elliptic estimates to control
derivatives of $\sigma$.

An important aspect of these elliptic estimates,
which are simple modifications of the estimates in \cite{CL00}
to the case of a non-Euclidean background metric, is that they only depend on
the regularity of the second fundamental form $\theta$ of the boundary
and do not rely on ``straightening the boundary'' is in the usual proof of
the elliptic estimates in Sobolev spaces.

We therefore need to get estimates for $\theta$. For this we
use another key idea from \cite{CL00}, which is that
$q = 0$ on $\pa \D_t$, then $\Pis\sd^r q$ can be controlled in terms of
$\sd^{r-1}q$ on the boundary, provided that the second fundamental
form $\theta$ is bounded. The basic idea is that if $q = 0$ on $\pa\D_t$
, $\Pis \sd q =0$ on $\pa \D_t$, and so:
\begin{equation}
  0 = \big(\Pis_i^k\sd_k\big)\big(\Pis^\ell_j \sd_\ell\big) q
  = \big(\Pis \sd^2\big)_{ij} q + \Pis_i^k\big(\sd_k \Pis_j^\ell\big)\sd_\ell q.
 \label{}
\end{equation}
Because $q = 0$ on $\pa \D_t$, the second term here is:
\begin{equation}
 \Pis_i^k\sd_k (\N_j \N^\ell) \N_\ell \sd_\N q = (\Pis_i^k\sd_k \N_j)\sd_N q,
 \label{}
\end{equation}
so we have:
\begin{equation}
 \big(\Pis \sd^2 \big)_{ij} q = \theta_{ij} \sd_\N q.
 \label{}
\end{equation}
In particular, $\Pis \sd^2 q$ is lower order. Further, if
$|\sd_\N q| > 0$ and we have a bound for $\Pis\sd^2 q$, this gives
a bound for $\theta$.

There is a higher-order version of this identity, which says that if
$q = 0$ on $\pa \D_t$ then:
\begin{equation}
 \Pis \sd^r q = (\nas^{r-2} \theta) \sd_\N q + O(\sd^{r-1} q, \sd^{r-3} \theta).
 \label{projbigo}
\end{equation}
See \eqref{projest}; a more precise version of this statement can be found in
Proposition 4.3 in \cite{CL00}. Using \eqref{projbigo}, we see that
if the sign condition \eqref{int:tsc} holds, the boundary
term in \eqref{int:ekl} controls $||\sd^{r-2} \theta||_{L^2(\pa \D_t)}$.
Using the above ideas, in section \ref{energysec}, we prove:
\begin{equation}
 ||V||_{H^r(\D_t)} + ||\nabla \sigma||_{H^{r}(\D_t)}
  + ||\theta||_{H^{r-2}(\D_t)} \leq C \E^r(t).
\end{equation}
See Lemma \ref{coerlem1} and Proposition \ref{coerprop1}.

\section{Definitions and assumptions}
\label{sec:defsec}
Throughout, we will use the convention that Greek letters
$\mu, \nu,...$ take values 0,1,2,3 and Latin letters $i,j,...$
take values 1,2,3. We sum over any repeated upper and lower indices,
and we will raise and lower indices with the metric $g$. We will write the
components of $g$ as $g_{\mu\nu}$ and the components of the inverse
metric as $g^{\mu\nu}$.

\subsection{Spatial derivatives}
\label{sec:spatsec}

We assume that $(M, g)$ admits a time function $\mathfrak{t}: M \to\R$ so
that if $M_t = \{\mathfrak{t} = t\}$, then the $M_t$ are Cauchy surfaces
and they foliate $M$.
We then define the future directed timelike unit (co)normal to the time
slice $M_t$:
\begin{align}
 &\tau_\mu  = \frac{1}{-\sqrt{g(\nabla \mathfrak{t},
 \nabla \mathfrak{t})}} \nabla_\mu \mathfrak{t} && \tau^\mu = g^{\mu\nu}
 \tau_\nu,
 \label{assump:taudef}
\end{align}
as well as the projection from $T(M)$ to $T(M_t)$:
\begin{equation}
 \sPi^\mu_\nu = \delta_\mu^\nu + \tau^\mu \tau_\nu.
 \label{}
\end{equation}
When $X$ is a four-vector we will write $\sX^\mu = \Pi^\mu_\nu X^\nu$
and $X^\tau = \tau^\mu X_\mu$.
If $(x^1, x^2, x^3)$ are local coordinates on $M_t$ and we write $t = x^0$,
then $X = X^\mu \pa_{x^\mu}$ where:
\begin{align}
 &\sX^0  =  0, &&\sX^i = X^i, i = 1,2,3,
 \label{}
\end{align}
and $X^\tau = \pm X^0$, with the minus sign when $X$ is future
directed and the plus sign when $X$ is past directed.

We let $\sg$ denote the induced Riemannian metric on $M_t$ extended to be
zero on the orthogonal complement of $T(M_t)$ in $T(M)$:
\begin{align}
 &\sg_{\mu\nu} = g_{\mu\nu} + \tau_\mu \tau_\nu, && \sg^{\mu\nu}
 = g^{\mu\nu}+
 \tau^\mu\tau^\nu.
 \label{}
\end{align}
We will also write:
\begin{align}
 &\gr_{\mu\nu} = \sg_{\mu\nu} + \tau_\mu \tau_\nu, &&
 \gr^{\mu\nu} = \sg^{\mu\nu}
 +
\tau^\mu \tau^\nu,
 \label{}
\end{align}
and:
\begin{equation}
 |X|^2 = \gr_{\mu\nu} X^\mu X^\nu = (X^\tau)^2 + \sg^{\mu\nu}X_\mu X_\nu.
 \label{}
\end{equation}
More generally, if $\beta$ is a $(0,r)$ tensor we write:
\begin{equation}
 |\beta|^2 = \gr^{\mu_1 \nu_1}\cdots \gr^{\mu_r\nu_r}\beta_{\mu_1\cdots\mu_r}
 \beta_{\nu_1 \cdots\nu_r}.
\end{equation}

The spatial derivatives of a four-vector $X$ are:
\begin{equation}
 \sd_\mu X^\nu = \sPi_\mu^{\mu'} \sPi^\nu_{\nu'} \nabla_{\mu'}X^{\nu'}.
 \label{sddef}
\end{equation}
Writing $X^\nu = -\tau^\nu X^\tau + \sX^\nu$, we see that:
\begin{equation}
 \sd_\mu X^\nu = -(\sd_\mu \tau^\nu) X^\tau + \sd_\mu \sX^\nu.
 \label{}
\end{equation}
In particular if $X$ is tangent to $M_t$, this agrees with the intrinsic
covariant derivative determined by $\sg$.

The divergence of a four-vector $X$ is defined as:
\begin{equation}
 \div X = \tr_g (\nabla X) =  \nabla_\mu X^\mu
 = \frac{1}{\sqrt{|\det g|}} \pa_\mu \big( \sqrt{|\det g|} X^\mu\big).
 \label{}
\end{equation}

To define the curl, we let $\zeta_X$ denote the one-form associated
to $X$:
\begin{equation}
 \zeta_X = (\zeta_X)_\mu dx^\mu = g_{\mu\nu}X^\nu dx^\mu,
 \label{}
\end{equation}
and then, if $d$ denotes the differential on $M$:
\begin{equation}
 \curl X = d \zeta_X = \big(\pa_\mu (g_{\mu'\nu} X^{\mu'})
 -(\pa_\nu \big(g_{\nu'\mu} X^{\nu'})\big)dx^\mu dx^\nu.
 \label{}
\end{equation}
When $X^\mu = \nabla^\mu \varphi$ for a function $\varphi:M\to\R$,
$\curl X = 0$.

There are analogous definitions of the divergence and curl intrinsic
to $M_t$. If $W \in T(M_t)$ then we write:
\begin{equation}
 \sdiv W = \tr_{\sg} \big( \sd \,W\big) =
 \sd_\mu W^\mu,
 \label{}
\end{equation}
and we extend this to four-vectors $X$ by setting $\sdiv X = \sdiv \sX$.

Writing:
\begin{equation}
 \bar{\zeta}_W =  (\bar{\zeta}_W)_i dx^i = \sg_{ij} W^jdx^i,
 \label{}
\end{equation}
where $(x_1,x_2,x_3)$ is any coordinate system on $M_t$,
we define:
\begin{align}
 \scurl W = \bar{d} \zeta_{W}
 = \big((\pa_i (\sg_{j\ell}W^\ell) - \pa_i (\sg_{i\ell}W^\ell)\big)dx^i
 dx^j,
 \label{}
\end{align}
where $\bar{d}$ is the differential on $M_t$. This is a two-form on
$M_t$ which we extend to a two-form on $M$ by setting:
\begin{equation}
 \scurl X(\tau,Y) = 0,
 \label{}
\end{equation}
for any vector $Y$.
We will also let $\sDelta$ denote the Laplace-Beltrami operator
determined by $\sg$:
\begin{equation}
 \sDelta q  = \sd_i \sd^i q
 = \sg^{ij} \sd_i \sd_j q = \sdiv \sd q.
 \label{assump:lb}
\end{equation}

We let $Rm$ denote the Riemann curvature tensor of $(M, g)$:
\begin{equation}
  Rm_{\mu\nu \alpha\beta}X^\beta = \nabla_\mu \nabla_\nu X_\alpha
  - \nabla_\nu \nabla_\mu X_\alpha,
\end{equation}
If $(x^0, x^1,x^2, x^3)$ are local coordinates on $M$ then:
\begin{equation}
 Rm_{\mu\nu\alpha}^\beta = \pa_\mu \Gamma^\beta_{\alpha \nu} -
 \pa_\alpha \Gamma^\beta_{\nu\mu}
 + \Gamma_{\nu\gamma}^\beta \Gamma^\gamma_{\alpha\mu} -
 \G^\beta_{\alpha\gamma}\G^\gamma_{\nu\mu},
 \label{}
\end{equation}
where the $\Gamma$ denote the Christoffel symbols of $g$ in this
coordinate system.

It is convenient to assume a bound of the form:
\begin{equation}
 \sum_{s = 1}^N |\nabla^s Rm| + |\nabla^s \tau| \leq R,
 \label{assump:rm}
\end{equation}
for some constant $R$ and sufficiently large $N$. A simple consequence that
we will use in section \ref{sec:ell} is that:
\begin{equation}
 \nabla_{\pi(I)} \beta = \nabla_I \beta + \mathcal{R}(\beta),
 \label{perm}
\end{equation}
where $I = (i_1,..., i_s)$ with $s \leq N$,
$\pi$ is a permutation on $s$ letters and $\mathcal{R}$
satisfies the estimates:
\begin{align}
 ||\mathcal{R}(\beta)||_{L^\infty(M_t)} \leq R\sum_{j \leq s-2}
 ||\nabla^j \beta||_{L^\infty(M_t)}, &&
 ||\mathcal{R}(\beta)||_{L^2(M_t)} \leq R\sum_{j \leq s-2}
 ||\nabla^j \beta||_{L^2(M_t)}.
 \label{}
\end{align}
It would not be difficult
to prove our results with milder assumptions on $Rm$ than
\eqref{assump:rm}, but this makes the statements and proofs of many
of our theorems simpler.

\subsection{The free boundary $\pa\D_t$}
We set $\D_t = \D \cap M_t$, and we let $\N = \N_i dx^i$
denote the exterior unit conormal to $\pa \D_t$ in $M_t$. We will
also write $N = N_\mu dx^\mu$ for the exterior unit conormal
to $\Lambda$.
We will use the metric $\sg$ to raise and lower indices $i,j,k,...$,
so that $\N^i = \sg^{ij}\N_j$.
The projection from $T(\D_t)$ to $T(\pa \D_t)$ at the boundary is:
\begin{equation}
 \Pis_i^j = \delta^i_j - \N_i \N^j,
 \label{}
\end{equation}
and the projection of an $(r,s)$ tensor $S$ to the boundary is:
\begin{equation}
 (\Pis S)_{i_1\cdots i_s}^{j_1\cdots j_r} = \Pis_{i_1}^{k_1}
 \cdots\Pis_{i_s}^{k_s}\Pis_{\ell_1}^{j_1}
 \cdots\Pis_{\ell_r}^{j_r}S_{k_1\cdots k_s}^{\ell_1\cdots\ell_r}.
 \label{}
\end{equation}

We let $\gamma$ denote the induced metric on $T(\pa \D_t)$,
extended to be zero on the orthogonal complement in $T(\D_t)$:
\begin{align}
 &\gamma_{ij} = \sg_{ij} - \N_i\N_j, && \gamma^{ij} = \sg^{ij} - \N^i\N^j.
 \label{}
\end{align}
We also write $\dvolbdy$ for the volume element on $\pa \D_t$.

We will use the following tangential spatial derivatives:
\begin{equation}
 \nas_i T^j = \Pis_i^{i'}\Pis_j^{j'} \sd_{i'} T^{j'}.
 \label{}
\end{equation}
When $T$ is tangent to $\pa \D_t$, this agrees with the intrinsic
covariant derivative on $\pa \D_t$ determined by the metric $\gamma$.

The second fundamental form of $\pa \D_t$ is:
\begin{equation}
 \theta_{ij} = (\Pis \sd \N)_{ij}.
 \label{}
\end{equation}
We let $\iota_0$ denote the normal injectivity radius of $\D_t$. By
definition this is the largest number $\iota_0$ so that the normal exponential
map:
\begin{equation}
 (-\iota_0, \iota_0) \times \pa \D_t \to \{x \in M_t |\,
 \textrm{dist}_{\sg}(x,\pa \D_t) < \iota_0\},
 \label{}
\end{equation}
defined by:
\begin{equation}
 (\iota, \omega) \mapsto \omega + \iota \N(\omega),
 \label{}
\end{equation}
is injective. Here, $\textrm{dist}_{\sg}$ denotes the geodesic distance with respect
to $\sg$.
We will assume that the following assumption on the boundary holds:
\begin{align}
 |\theta| + \frac{1}{\iota_0} \leq K \textrm{ on } \pa \D_t.
 \label{assump:bdy}
 \end{align}

 \subsection{Sobolev spaces}
 For a $(0,r)-$tensor $T$ on $\D$, recall that we are writing:
 \begin{equation}
  |T|^2 = \gr^{\mu_1\nu_1} \cdots\gr^{\mu_r \nu_r}T_{\mu_1 \cdots  \mu_r}
  T_{\nu_1\cdots \nu_r},
  \label{}
 \end{equation}
 where $\gr$ is the Riemannian metric $\gr_{\mu\nu} = \sg_{\mu\nu} +
 \tau_\mu \tau_\nu$.
 We define:
 \begin{align}
  ||T||_{L^2(\D_t)}^2 =
  \int_{\D_t}|T|^2 \dvols && ||T||_{H^r(\D_t)}^2
  = \sum_{\ell = 0}^r ||\sd^\ell T||_{L^2(\D_t)}^2,
  \label{}
 \end{align}
 as well as:
 \begin{align}
  ||T||_{L^2(\pa \D_t)}^2 &= \int_{\pa \D_t} |T|^2 \dvols,
  &&||T||_{H^r(\pa \D_t)}^2 = \sum_{\ell = 0}^r
   ||\nas^\ell T||_{L^2(\pa \D_t)}.
  \label{}
 \end{align}

\subsection{Assumptions on the fluid variables}
\label{fluidassump}

By definition, the fluid velocity $u$ is a future-directed timelike unit vector
field:
\begin{align}
 &u^\mu u_\mu = -1, &&u^\tau < 0.
 \label{unit}
\end{align}
In particular note that $\utau \geq 1$ with equality only
when $\ubar = 0$.
With $\lambda = \frac{|\ubar|}{|\utau|}$, we will assume that:
\begin{equation}
  \lambda \leq \lambda^*,
 \label{assump:stat}
\end{equation}
for some $\lambda^*$ which will be chosen sufficiently small.
By \eqref{unit}, we always have $\lambda < 1$.
We now explain why we need this assumption.
First, in \eqref{ellpw} we will estimate $\sd V$ in terms
of $\sdiv V, \scurl V$ and tangentially projected derivatives
$\Pis \sd V$. The equation \eqref{int:dtesig} gives us an equation
for $\div V$, and by \eqref{sdivident}, we have:
\begin{equation}
 \sdiv V = \div V + \tau_\mu\nabla_\tau V^\mu.
 \label{}
\end{equation}
There is a similar formula relating $\scurl V$ and $\curl V$,
and so to close the estimates we must control $\nabla_\tau V^\mu$.
By definition:
\begin{equation}
  \nabla_\tau = \frac{1}{\utau} \big( \nabla_u - \ubar^\mu \nabla_\mu\big).
  \label{matderivident}
\end{equation}
This leads to an estimate of the form:
\begin{equation}
 |\sd V| \leq C \big( |\sdiv V| + |\scurl V| + |\nabla_u V| +
 \lambda |\sd V| + |\Pis \sd V|\big).
 \label{}
\end{equation}
We have equations for $\sdiv V, \scurl V$ and $\nabla_uV$, and
the energy bounds the last term. If
\eqref{assump:stat} holds, we are able to absorb the fourth term into the
left-hand side.
See Lemma \ref{difflem}.

We also use
\eqref{assump:stat} to control $\sigma$ in section \ref{wave}. By
the wave equation \eqref{int:wave} and using the fact that
$\nabla_V$ preserves boundary conditions (by \eqref{int:free2}),
it is straightforward to prove estimates for $\nabla_V^\ell \sigma$.
To control space derivatives,
in \eqref{ellform1} we re-write \eqref{int:wave}
as an elliptic
equation of the form:
\begin{equation}
 \sDelta \sigma = e'(\sigma) \nabla_V^2 \sigma + \nabla_\tau^2 \sigma + ...
 \label{}
\end{equation}
We can then use the elliptic estimates from
\cite{CL00} to control spatial derivatives. We will control
$\nabla_V^2\sigma$ directly as above but we again need \eqref{assump:stat}
to deal with the pure time derivatives of $\sigma$.

In \cite{Oliynyk2017}, Oliynyk has also proved energy estimates for the
system \eqref{int:dtv}-\eqref{int:free2} using a different approach which avoids
the assumption \eqref{assump:stat}.
There, the equations \eqref{int:dtv}-\eqref{int:dtesig} and boundary conditions
\eqref{int:liquidsig}-\eqref{int:free2} are differentiated once in time
and then reformulated as a system of wave equations with
``acoustic'' boundary conditions. In Lagrangian coordinates,
one can use the ``standard'' elliptic estimate in fractional Sobolev
spaces (see Theorem B.4 in \cite{Oliynyk2017}, and
Proposition B.3.1 in \cite{Nordgren2008} for similar estimates)
to get control over, in our notation, $||V||_{H^r(\D_t)} +
||\sigma||_{H^{r+1}(\D_t)} +
||x||_{H^{r+1/2}(\pa \D_t)}$, where $x$ denotes the Lagrangian coordinate
 map. This is different from our energy, where one instead controls
 the second fundamental form on the boundary in Sobolev spaces with
 an integer number of derivatives.

We will be considering a barotropic
fluid, so that $p = p(\epsilon)$ for some smooth strictly increasing
function $p$. We assume that:
\begin{align}
 \bigg|\frac{d^k}{d \epsilon^k} p(\epsilon)\bigg| \leq L_1,
 \,\,  k = 1,..., N,
 \label{assump:baro}
\end{align}
for sufficiently large $N$ to be determined later.
The quantity $\eta(\epsilon) \equiv \sqrt{p'(\epsilon)}$ is called
the sound speed, and we will assume that there is a constant
$L_2$ so that:
\begin{align}
  &0 < L_2 \leq \eta^2 \leq 1. \label{assump:sound}
\end{align}
Recall that we are working in units so that the speed of light is 1, so this
says the sound speed cannot exceed the speed of light.

Using \eqref{int:therm} and the fact that $p$ is an invertible
function of $\epsilon$, given any one of
$p,\rho, \epsilon$, we can determine the other two. As mentioned in the
introduction, we will work
in terms of the enthalpy $\sigma$, which is defined by:
\begin{equation}
 \sqrt{\sigma} = \frac{\epsilon + p}{\rho}.
 \label{}
\end{equation}
We also need to assume a lower bound for $\sigma$ in
$\D_t$:
\begin{equation}
 0 < L_3 \leq \sigma.
 \label{assump:lower}
\end{equation}
By \eqref{int:liquidsig}, this holds automatically near $\pa \D_t$ with
$L_3 = \sigma_0/2$, say.

With $e(\sigma) = \log (\rho(\sigma)/\sqrt{\sigma})$, we assume:
\begin{align}
  &\bigg|\frac{d^k}{d\sigma^k} e(\sigma)\bigg|
  \leq L_3 |e(\sigma)|, && \textrm{ for } k = 1,..., N,
  \label{assump:ecoer}
\end{align}
for sufficiently large $N$. We write $L = L_1 + 1/L_2 + 1/L_3 + L_4$.

An interesting equation of state satisfying the above conditions
is the ``two-phase model'':
\begin{equation}
  p(\epsilon) =
  \begin{cases}
 \epsilon - \epsilon_0,&\epsilon > \epsilon_0,\\
 0,  &\epsilon \leq \epsilon_0
\end{cases}
 \label{}
\end{equation}
which was introduced in \cite{Christodoulou1995a} as a model for
the collapse of a spherically symmetric neutron star.
(Note that Christodoulou uses $\rho, n$ to denote the energy
density and mass density, respectively).
This corresponds to sound speed $\eta = 1$ in our notation, which
is the largest physically realistic value. Such a model is known
as a ``stiff'' or ``incompressible'' fluid;
see \cite{Lichnerowicz1967a}, \cite{Rezzolla2013}. See also the
recent work \cite{Fournodavlos2017} where static solutions
to the Einstein-Euler system \eqref{int:emtensor}
with this equation of state are studied.

We will also
need to assume some bounds for derivatives of $u$ and $\sigma$ in $\D$:
\begin{align}
  |\nabla u| + |\nabla \sigma|  \leq M_1, \textrm{ in } \D_t,
  \label{assump:fluid1}\\
  |\nabla^2 \sigma| \leq M_2 \textrm{ on } \pa \D_t,
 \label{assump:fluid2}
\end{align}
and we will write $M = M_1 + M_2$.

As mentioned in the introduction, we will assume that the following
condition holds:
\begin{equation}
 -\sd_{\N} p \geq \delta  >0 \textrm{ on } \pa\D_t.
 \label{assump:tsc}
\end{equation}
In the non-relativistic setting, the equations \eqref{int:consu}-\eqref{int:free}
are ill-posed unless \eqref{assump:tsc} holds; see \cite{Ebin1987}.
Note that by \eqref{int:free2}, writing
$|\nabla p|_g^2 = g(\nabla p, \nabla p)$, the condition \eqref{assump:tsc}
is equivalent to $|\nabla_N p| \geq C \delta > 0$ for some constant
 $C = C(\lambda)$.
By \eqref{int:therm},
$\frac{d}{d p} \sqrt{\sigma} = \frac{1}{\epsilon + p}$,
so if \eqref{liquideps}-\eqref{liquid} and \eqref{assump:tsc} hold,
\begin{equation}
 -\nabla_N \sigma \geq  \delta' \equiv C\frac{\rho_0}{\sigma_0}\delta > 0.
 \label{assump:tsc2}
\end{equation}

\section{Lagrangian coordinates, the extension of the normal}

We now introduce a system of coordinates on $\D_t$, known
as Lagrangian coordinates,
 which fixes the boundary
$\pa \D_t$. Although all of our results are invariant under coordinate changes,
it is convenient to prove some of the estimates in Lagrangian coordinates.
In addition,
we will need to use many of the results from \cite{CL00}
which are stated in Lagrangian coordinates.
Let $\Omega$ be a domain in $\R^3$ and let $f:\D_0 \to \Omega$
be a volume preserving diffeomorphism.
If $V$ satisfies \eqref{int:free}, we define the following vector field on $\D_t$:
\begin{equation}
 W^i = \frac{1}{-V^\tau} \overline{V}^i,
\end{equation}
Then the Lagrangian coordinates are
a mapping $x: [0,T] \times\Omega \to \D$ defined by:
\begin{equation}
 \frac{d}{dt} x(t, y) = W(t, x(t,y)).
 \label{lagdef}
\end{equation}
For each $t$, the map $x(t, \cdot)$ is a diffeomorphism from
$\Omega$ to $\D_t$, and note that $x(t, \pa \Omega) = \pa \D_t$ by
\eqref{int:free}.

We will use the letters $a,b,c...$ to denote quantities expressed in the
$y$-coordinates, and we will abuse notation and use $\sg$ to also denote the
induced Riemannian metric on $\Omega$:
\begin{equation}
 \sg_{ab} = \sg_{ij}\frac{\pa y^a}{\pa x^i}\frac{\pa y^b}{\pa x^j}.
 \label{}
\end{equation}
Similarly we write:
\begin{align}
 \gamma_{ab} &= \gamma_{ij} \frac{\pa y^a}{\pa x^i}\frac{\pa y^b}{\pa x^j},
 && \Pis^a_b = \sg^{ac}\gamma_{bc}
 \label{gamdef}
\end{align}

We now extend the projection $\Pis$ to a neighborhood of the boundary.
Let $d(y) = \dist(y, \pa \Omega)$ denote the geodesic
distance to the boundary, and let $\varphi$ be a function on
$\Omega$ so that $\varphi = 0$ on $\pa \Omega$, $\varphi < 0$ in $\Omega$,
and so that $|\sd \varphi| > 0$ on $\pa \Omega$.
Then the outer conormal and normal to $\pa \Omega$ are given by:
\begin{align}
 \widetilde{\N}_a = \frac{\pa_a \varphi}{\sg(\nabla \varphi, \nabla \varphi)},
 && \widetilde{\N}^a = \sg^{ab}\widetilde{\N}_b.
 \label{extndef}
\end{align}
Let $\iota_0$ be as in \eqref{assump:bdy} and
let $\chi = \chi(d)$ denote a smooth
positive function so that $\chi = 1$ when $d \leq \iota/4$ and
so that $\chi = 0$ when $d > \iota/2$. We then extend the projection $\Pis$
(defined in \eqref{Pisdef}) to the interior by defining:
\begin{equation}
 \widetilde{\Pis}^a_b = \delta^a_b - \chi(d) \widetilde{N}^a \widetilde{N}_b.
 \label{extPidef}
\end{equation}
Away from the boundary this is the identity map on $T(\Omega)$
and on $\pa \Omega$ this is the projection to $T(\pa \Omega)$. From
now on we will write \eqref{extPidef} as $\Pis$.
We will also let $\gamma^{ab}$ denote the metric \eqref{gamdef} extended to the
interior as in \eqref{extPidef}.

The properties of the projection and its extension to the interior that
we will use are:
\begin{lemma}
  Let $h_{ab} = \nabla_V \sg_{ab}$. On $[0,T]\times \pa \Omega$:
 \begin{align}
  \nabla_V \gamma^{ab} = -\gamma^{ac}\gamma^{bd}h_{cd},
  \label{dtgam}
 \end{align}
 and:
 \begin{align}
  ||\sd \gamma||_{L^\infty(\Omega)} \leq C\bigg( ||\theta||_{L^\infty(\pa \Omega)}
  + \frac{1}{\iota_0}\bigg), &&
  ||\nabla_V \gamma||_{L^\infty(\Omega)} \leq C ||h||_{L^\infty(\Omega)}.
  \label{dgam}
 \end{align}
\end{lemma}
\begin{proof}
 See Lemma 3.11 of \cite{CL00}.
\end{proof}

We will write $Vol(\Omega)$ for the volume of $\Omega$ with respect to
the metric $\sg$. The following lemma can be used to control $Vol(\Omega)$:
\begin{lemma}
  With the above definitions, on $[0, T] \times \Omega$:
 \begin{equation}
  \frac{d}{dt} \det \sg = \frac{\det \sg}{V^0} \big(\sdiv V -
  V^i\pa_i \log V^0).
  \label{dtvol}
 \end{equation}
\end{lemma}
\begin{proof}
  We abuse notation
  and write $V(t,y) = V(t,x(t,y))$ with $x$ from \eqref{lagdef}.
  By the well-known formula for the derivative of the determinant:
  \begin{equation}
    \frac{d}{dt} \det \sg(t, y) = \det \sg \,\sg^{ab}(t,y) \frac{d}{dt}\sg_{ab}(t,y),
   \label{jacobi}
  \end{equation}
  \end{proof}
  and by \eqref{lagdef}, we have:
  \begin{equation}
   \frac{d}{dt} \sg_{ab}(t,y) = \bigg(\frac{d}{dt} \sg_{ij}(t,y)\bigg)
   \frac{\pa x^i(t,y)}{\pa y^a} \frac{\pa x^j(t,y)}{\pa y^b}
   + 2\sg_{ij}(t,y) \frac{\pa }{\pa y^a} \bigg(\frac{V^i(t,y)}{V^0(t,y)}\bigg)
   \frac{\pa x^j(t,y)}{\pa y^b}.
   \label{}
  \end{equation}
  The equation \eqref{jacobi} then implies:
  \begin{equation}
   \frac{d}{dt} \sg(t,y) = \frac{\det \sg(t,y)}{V^0(t,y)}
    \bigg(\pa_i V^i(t,y) - \frac{V^i(t,y)}{V^0(t,y)} \pa_i V^0(t,y)\bigg).
   \label{}
  \end{equation}

\section{Elliptic Estimates}
\label{sec:ell}

We now fix a $(0,1)$ tensor $\alpha_k$ on $\D_t$ and let $\beta_k = \beta_{Ik}
 = \sd_I^r \alpha_k$, where $\sd_I^r =
\sd_{i_1}\cdots \sd_{i_r}$.
We will also let $\beta^S_{Ik}$ and $\beta^A_{Ik}$
denote the symmetrization and anti-symmetrization of $\beta$ over
the first $r$ indices respectively:
\begin{align}
 \beta^S_{i_1\cdots i_r k} = \frac{1}{r!}\sum_{\pi \in S_r}
 \beta_{i_{pi(1)}\cdots i_{\pi(r)} k} &&
 \beta^A_{i_1\cdots i_r k} = \frac{1}{r!}\sum_{\pi \in S_r}
 (-1)^{|\pi|} \beta_{i_{\pi(1)}\cdots i_{\pi(r)} k}.
 \label{ell:symdef}
\end{align}
Here $S_r$ is the symmetric group on $r$ letters and $|\pi|$ is
the order of the permutation $\pi$. Then $\beta = \beta^S + \beta^A$
and $|\beta^S| + |\beta^A| \leq 2|\beta|$.
When $\sg$ is flat
then $\beta^S = \beta$ and $\beta^A = 0$
but in general $\beta^A$ will involve derivatives of the Riemann curvature
tensor.

We will rely on several elliptic estimates from \cite{CL00},
where it is assumed that $[\sd_i, \sd_j] = 0$. To apply these estimates
to the general case, we will
write $\beta = \beta^S + \beta^A$ and many of the results from that paper
can then be applied directly to $\beta^S$. Since we are assuming that
we have bounds \eqref{assump:rm} for the Riemann tensor, $\beta^A$ will be
lower order.

We will write:
\begin{align}
 &\sdiv \beta = \sg^{ij}\sd_i \beta_{Ij},
 &&\scurl \beta_{ij} = \nabla_i \beta_{Ij} - \nabla_j
 \beta_{Ii}.
 \label{}
\end{align}
Note that by \eqref{assump:rm} we have:
\begin{align}
 \sdiv \beta = \sd^r \sdiv \alpha + R_1\alpha,&&
 \scurl \beta = \sd^r \scurl \alpha + R_2\alpha,
 \label{ell:symident}
\end{align}
where
\begin{equation}
 |R_1\alpha| + |R_2\alpha|
  \leq C(R) \sum_{k \leq r-1} |\sd^{k}\alpha|,
 \label{}
\end{equation}
if $r \leq N$.

If $I = (i_1,\cdots, i_r), J = (j_1,\cdots j_r)$, we will write:
\begin{align}
 \sg^{IJ} = \sg^{i_1j_1}\cdots \sg^{i_rj_r}, &&
 \gamma^{IJ} = \gamma^{i_1j_1}\cdots \gamma^{i_rj_r}.
 \label{}
\end{align}

The fundamental elliptic estimates are then:
\begin{lemma}
 If $\beta$ is as above
  and the assumption \eqref{assump:bdy} holds, then:
 \begin{align}
  |\sd\, \beta|^2 &\leq C\bigg( \sg^{ij}\gamma^{k\ell} \gamma^{IJ}
  \sd_k \beta_{Ii} \sd_\ell \beta_{Jj}
  + |\sdiv \beta|^2 + |\scurl \beta|^2 + |\sd \beta^A|^2 \bigg),
  \label{ellpw}\\
  \int_{\Omega} |\sd \, \beta|^2 \dvols
  &\leq C \int_{\Omega} \big(\N^i \N^j \sg^{k\ell}\gamma^{IJ}
  \sd_k \beta_{Ii} \sd_\ell \beta_{Jj} + |\sdiv \beta|^2 +
   |\scurl \beta|^2
  + K |\beta|^2 + |\sd\beta^A|^2 \big)\dvols.
  \label{ellfund2}
 \end{align}
\end{lemma}
\begin{proof}
  Write $\beta = \beta_S + \beta_A$ and apply the proof of lemma 5.5 from
  \cite{CL00} to $\beta_S$.
\end{proof}

\begin{prop}
 With $\beta$ as defined above, if \eqref{assump:bdy} and
 \eqref{assump:rm} holds with $N \geq 1$  and $|Ric(g)| \leq R'$, then:
 \begin{equation}
  ||\beta||_{L^2(\pa \Omega)}^2 \leq C \big( ||\sd \beta||_{L^2(\Omega)}
  + K||\beta||_{L^2(\Omega)}\big)||\beta||_{L^2(\Omega)},\label{ellbdy1}
\end{equation}
\begin{multline}
  ||\beta||_{L^2(\pa \Omega)}^2 \leq C ||\Pis \beta||^2_{L^2(\pa \Omega)}
  + ||\beta^A||_{L^2(\pa \Omega)}^2\\
  + C\big( ||\sdiv \beta||_{L^2(\Omega)} + ||\scurl \beta||_{L^2(\Omega)}
  + K||\beta||_{L^2(\Omega)}\big)||\beta||_{L^2(\Omega)}\label{ellbdy2},
\end{multline}
 and
 \begin{multline}
  ||\sd \,\beta||_{L^2(\Omega)}^2 \leq C||\sd \beta||_{L^2(\pa \Omega)}
  ||\beta||_{L^2(\pa \Omega)} \\+C\big(||\sdiv \beta||_{L^2(\Omega)}
  + ||\scurl \beta||_{L^2(\Omega)} \big)^2
  + ||Ric(\sg)||_{L^\infty(\Omega)}||\beta||_{L^2(\Omega)}^2  \label{ellint1},
 \end{multline}
 \begin{multline}
  ||\sd \,\beta||_{L^2(\Omega)}^2
  \leq C||\Pis \sd \,\beta||_{L^2(\pa \Omega)}
  ||\Pis \N \cdot \beta||_{L^2(\pa \Omega)}\\
  + C(||\sdiv \beta||_{L^2(\Omega)} + ||\scurl \beta||_{L^2(\Omega)}
  + (K + R)||\beta||_{L^2(\Omega)}^2) +
  ||\sd \beta^A||_{L^2(\Omega)}^2.
  \label{ellint2}
 \end{multline}
 Here, $Ric(\sg)$ denotes the Ricci curvature of $(\Omega, \sg)$
 and $\N\cdot \beta = \N^\mu \beta_{I\mu}$.
\end{prop}
\begin{proof}
  This is a straightforward modification of Lemma 5.6 from
  \cite{CL00} and we just indicate the main points.
  First, \eqref{ellbdy1} follows by Stokes' theorem:
 \begin{equation}
  \int_{\pa \Omega} \gr^{\mu\nu}\sg^{IJ} \beta_{I\mu} \beta_{J\nu} \dvolbdy
  = \int_{\Omega} \sd_k\big(\N^k \gr^{\mu\nu} \sg^{IJ} \beta_{I\mu}
  \beta_{J\nu}) \dvol.
  \label{}
 \end{equation}
 The estimate \eqref{ellbdy2} is proved by writing $\beta = \beta^S + \beta^A$
 and applying the proof of (5.20) from \cite{CL00} to $\beta^S$.

 To prove \eqref{ellint1} we use the ``Hodge'' identity:
 \begin{align}
  \sDelta \beta_k = \sd^i\sd_i \beta_k
  &= \sd^i \sd_k \beta_i + \sd^i(\sd_i \beta_k - \sd_k \beta_i)\\
  &= \sd_k \sdiv \beta + \sd^i \scurl \beta_{ik}
  + Ric_{k\ell}\beta^\ell,
  \label{}
 \end{align}
 and then integrate by parts twice.

 Finally, to prove \eqref{ellint2}, we start with \eqref{ellfund2}
 and write $\N^i \N^j g^{k\ell}  = \N^i \N^k g^{j\ell}
 + A^{ijk\ell}$ with $A^{ijk\ell} =
 \N^i \N^j g^{k\ell} - \N^i \N^k g^{j\ell}$. We now modify
 the proof of (5.11) from \cite{CL00}. Note that:
 \begin{multline}
  A^{ijk\ell}\gamma^{IJ} \sd_k\beta_{Ii} \sd_\ell \beta_{Jj}\\
  = \sd_k (A^{ijk\ell} \gamma^{IJ} \beta_{Ii} \sd_{\ell}\beta_{Jj})
  - \sd_k(A^{ijk\ell})\gamma^{IJ}\beta_{Ii}\sd_\ell \beta_{Jj}
  - A^{ijk\ell} \gamma^{IJ}\beta_{Ii} \big(\Rm_{k\ell j}^m \beta_{Jm}\big),
  \label{ellibp}
 \end{multline}
 where $\Rm$ is the Riemann curvature of $(\Omega, \sg)$, and
 in the last term we used that $A^{ijk\ell} = -A^{ikj\ell}$. Also
 note that on $\pa \Omega$:
 \begin{equation}
  \N_k A^{ijk\ell}
  = \N_k(\N^i\N^j g^{k\ell} - \N^i\N^k g^{j\ell})
  = -\N^i \gamma^{j\ell}.
  \label{}
 \end{equation}
 Integrating \eqref{ellibp} over $\Omega$, using Stokes' theorem,
  this identity gives:
 \begin{multline}
  \int_{\Omega} A^{ijk\ell}\gamma^{IJ} \sd_k\beta_{Ii}
  \sd_\ell \beta_{Jj}\dvols\\
  \leq ||\Pis \sd\beta||_{L^2(\pa \Omega)}
  ||\Pis (\N\cdot \beta)||_{L^2(\pa \Omega)}
  +\big(K ||\sd \beta||_{L^2(\Omega)} + R||\beta||_{L^2(\Omega)}\big)
  ||\beta||_{L^2(\Omega)}\\
  + C\big(||\sdiv \beta||_{L^2(\Omega)}^2 + ||\scurl\beta||_{L^2(\Omega)}^2
  + K ||\beta||_{L^2(\Omega)}^2\big).
  \label{}
 \end{multline}
 Combining this with \eqref{ellfund2} then implies \eqref{ellbdy2}.
 \end{proof}

 The above estimates applied to $\beta = \sd^\ell q$ for a function $q$
 give estimates that
 we will need to control solutions to the wave equation in the next section:
\begin{prop}
 Suppose that the assumptions \eqref{assump:bdy} hold and \eqref{assump:rm} holds
 with $N \geq r$.
 If $q: \Omega \to\R$ and $r \geq 2$:
 \begin{multline}
  ||\sd^r q||_{L^2(\pa \Omega)}^2 + ||\sd^r q||_{L^2(\Omega)}^2
  \leq C ||\Pis \sd^r q||_{L^2(\pa \Omega)}^2 \\+ C(K, R)
  \bigg(
  ||\sDelta q||_{H^{r-1}(\Omega)}^2 + ||\sd q||_{L^2(\Omega)}^2
  + ||q||_{L^2(\Omega)}^2 + ||q||_{L^2(\pa \Omega)}^2\bigg).
  \label{ellbdyfn1}
\end{multline}
Furthermore, for any $\delta > 0$:
\begin{multline}
  ||\sd^r q||_{L^2(\Omega)}^2 + ||\sd^{r-1} q||_{L^2(\pa \Omega)}^2
  \leq C \delta ||\Pis \sd^r q||_{L^2(\pa \Omega)}^2\\
  + C(K,R,1/\delta) \bigg(||\sDelta q||_{H^{r-2}(\Omega)}^2 +
   ||\sd q||_{L^2(\Omega)}^2 + ||q||_{L^2(\Omega)}\bigg).
  \label{ellbdyfn2}
\end{multline}
\end{prop}
\begin{proof}
This is a simple modification of the proof of Proposition 5.8
of \cite{CL00}, which relies on the
previous proposition. The only ingredients needed are that,
  by \eqref{assump:rm},
  \begin{equation}
   |(\sd^r q)^A|
   \leq R \sum_{k \leq r-2} |\sd^k q|,
   \label{}
  \end{equation}
  and that by \eqref{ell:symident} we also have:
  \begin{equation}
   ||\sdiv \sd^{r} q||_{L^2(\Omega)}
   + ||\scurl \, \sd^r q|||_{L^2(\Omega)}
   \leq ||\sd^{r-1}\sDelta q||_{L^2(\Omega)} + R ||q||_{H^{r-2}(\Omega)}
   \label{}
  \end{equation}

  \end{proof}

  We will also need the following estimate
  for $|| \Pis \sd^rq||_{L^2(\pa \Omega)}$ when $\nas q = 0$
  on $\pa \Omega$, which is a simple modification of
  Proposition 5.9 from \cite{CL00}:
  \begin{prop}
    \label{projprop}
    Suppose that $|\theta| + 1/\iota_1 \leq K$.
    If $\nas q = 0$ on $\pa \Omega$, then for $m = 0,1$:
    \begin{multline}
     ||\Pis \sd^r q||_{L^2(\pa \Omega)}
     \leq C(K,R)\bigg( ||(\nas^{r-2} \theta) (\sd_\N q)||_{L^2(\pa \Omega)}
     + \sum_{k = 1}^{r-1} ||\sd^{r-k} q||_{L^2(\pa \Omega)}\\
     + \big(||\theta||_{L^\infty(\pa \Omega)} + ||\theta||_{H^{r-3}(\pa\Omega)}
     \big) \sum_{k = 0}^{r-1} ||\sd^k q||_{L^2(\pa \Omega)}\bigg),
     \label{projest}
    \end{multline}
    and if $r > 3$, for any $\delta > 0$:
    \begin{equation}
     ||\Pis \sd^{r-1} q||_{L^2(\pa \Omega)} \leq
     \delta||\sd^{r-1} q||_{L^2(\pa \Omega)} +
     C(1/\delta, K, ||\theta||_{H^{r-3}(\pa \Omega)})
     \sum_{k = 0}^{r-2} ||\sd^k q||_{L^2(\pa \Omega)}.
     \label{}
    \end{equation}

    In addition, if $|\sd_\N q| \geq \delta > 0$ and $|\sd_\N q| \geq 2\delta
    ||\sd_\N q||_{L^\infty(\pa \Omega)}$ then:
    \begin{multline}
     ||\nas^{r-2} \theta||_{L^2(\pa \Omega)}
     \leq C(1/\delta, K, R) \bigg(||\Pis \sd^r q||_{L^2(\pa \Omega)}
     + \sum_{k =1}^{r-1}
     ||\sd^{r-k}q||_{L^2(\pa \Omega)}\\
     +  \big(||\theta||_{L^\infty(\pa \Omega)}
     + ||\theta||_{H^{r-3}(\pa \Omega)}\big)
     \sum_{k \leq r-1} ||\sd^k q||_{L^2(\pa \Omega)}\bigg),
     \label{projtheta}
    \end{multline}
    where the last lines in \eqref{projtheta}
    and \eqref{projest} are not present if $r \leq 4$.
  \end{prop}
    \begin{proof}
     Proposition 5.9 of \cite{CL00} is stated for
     $q = 0$ on $\pa\Omega$ but what is
     actually used is that $\nas^r q = 0$ (see inequality
     (4.5) there).
     As in the previous lemmas, we can apply the proof that proposition
     to $(\sd^r q)^S$, and
     bound $||(\sd^r q)^A||_{L^2(\pa \Omega)}$ by \eqref{perm}.
     The estimate \eqref{projtheta} is similar.
    \end{proof}

    The final result we need is a
    version of
    Proposition 5.8 from \cite{CL00}:
    \begin{prop}
      \label{elllotbdy}
      For any $0 \leq r \leq 4$, if $\sd q = 0$ on $\pa \Omega$,
      then:
      \begin{multline}
       ||\sd^{r-1}q||_{L^2(\pa\Omega)}
       \leq C\big( ||(\nas^{r-3} \theta)(\sd_\N q)||_{L^2(\pa\Omega)}
       + ||\sd^{r-2}\sDelta q||_{L^2(\Omega)})\\
       + C(K, Vol(\Omega), R, ||\theta||_{H^{r-4}(\pa \Omega)}
       \bigg(||\sd_N q||_{L^\infty(\pa \Omega)} +
        ||\sDelta q||_{H^{r-3}(\Omega)}\bigg),
       \label{elllotbdy1}
      \end{multline}
      and if $r > 3$:
      \begin{multline}
       ||\sd^{r-1} q||_{L^2(\pa \Omega)} + ||\sd q||_{L^\infty(\pa \Omega)}
       \leq C ||\sd^{r-2} \sDelta q||_{L^2(\Omega)}\\
       + C(K, Vol(\Omega), R, ||\theta||_{L^2(\pa \Omega)} ,
       ..., ||\nas^{r-3} \theta||_{L^2(\pa \Omega)})
        ||\sDelta q||_{H^{r-3}(\Omega)}.
       \label{elllotbdy2}
      \end{multline}

    \end{prop}
    \begin{proof}
      By \eqref{ellbdyfn1} with $m = 1$
      and the estimate \eqref{projest},
      \begin{multline}
       ||\sd^{r-1} q||_{L^2(\pa \Omega)}
       \leq C(K)\bigg(||\sd^{r-3} \theta)(\sd_\N q)||_{L^2(\pa \Omega)}
       + ||\sDelta q||_{H^{r-2}(\Omega)}\\
       + \big(||\theta||_{L^\infty(\pa \Omega)} +
       ||\theta||_{H^{r-4}(\pa \Omega)} \big)\sum_{k = 0}^{r-2}
       ||\sd^k q||_{L^2(\pa \Omega)}^2
       + ||\sd q||_{L^2(\Omega)}^2 + ||q||_{L^2(\Omega)}^2 +
       ||q||_{L^2(\pa \Omega)}^2\bigg).
       \label{}
      \end{multline}
      Because $q = 0$ on $\pa \Omega$, the last term vanishes and
      using \eqref{poin},\eqref{poin2} the third and second last terms are bounded by
      $C(Vol \Omega) ||\sDelta q||_{L^2(\Omega)}$. Repeating this
      argument and using \eqref{bdysob2}
      gives \eqref{elllotbdy1} and \eqref{elllotbdy2}.
    \end{proof}

\section{Estimates for the acoustic wave equation}
\label{sec:wave}

Taking the divergence of \eqref{int:dtv} gives:
\begin{equation}
  \nabla_\mu \nabla_V V^\mu  + \frac{1}{2} g^{\mu\nu}
  \nabla_\mu \nabla_\nu \sigma
= \nabla_V \nabla_\mu V^\mu + (\nabla_\mu V^\nu) (\nabla_\nu V^\mu)
+ \frac{1}{2}g^{\mu\nu} \nabla_\mu \nabla_\nu \sigma,
  \label{}
\end{equation}
while applying $\nabla_V$ to \eqref{int:dtesig} gives:
\begin{equation}
  \nabla_V^2 e(\sigma) + \nabla_V \nabla_\mu V^\mu = 0.
  \label{}
\end{equation}
Therefore, $\sigma$ satisfies:
\begin{equation}
  \nabla_V^2 e(\sigma) - \frac{1}{2}g^{\mu\nu}\nabla_\mu \nabla_\nu \sigma =
  (\nabla_\mu V^\nu)(\nabla_\nu V^\mu).
  \label{wave}
\end{equation}
In this section, it is more convenient to work with $\nabla_u$ than
of $\nabla_V$. Recalling that $V^\mu = \sqrt{\sigma} u^\mu$, the above is:
\begin{equation}
  \sigma \nabla_u^2 e(\sigma) - \frac{1}{2}g^{\mu\nu}\nabla_\mu \nabla_\nu
  \sigma
  = (\nabla_\mu V^\nu)(\nabla_\nu V^\mu) - (\nabla_u e(\sigma))
  (\nabla_u \sigma)
  \label{}
\end{equation}

By \eqref{int:therm} and the definitions $e(\sigma) =
\log(\rho(\sigma)/\sqrt{\sigma})$,
$\eta = \sqrt{p'(\epsilon)}$,
we have:
\begin{equation}
  e'(\sigma) = \frac{1}{2\sigma} \big(\eta^{-2} - 1\big),
  \label{}
\end{equation}
so the first term above is:
\begin{align}
  \sigma \nabla_u^2 e(\sigma) &=
  \sigma e'(\sigma) \nabla_u^2 \sigma + \sigma (\nabla_u e'(\sigma))
  \nabla_u \sigma\\
  &= \frac{1}{2} \big(\eta^{-2} - 1\big)\nabla_u^2 \sigma
  + (\nabla_u e'(\sigma))(\nabla_u \sigma).
  \label{}
\end{align}
The equation \eqref{wave} can be then written as:
\begin{equation}
  (\eta^{-2} - 1) \nabla_u^2 \sigma -
  \nabla_\mu\big(g^{\mu\nu} \nabla_\nu \sigma\big) = 2(\nabla_\mu V^\nu)(\nabla_\nu V^\mu)
  -(\nabla_u e(\sigma) + \nabla_u e'(\sigma)) \nabla_u \sigma
  \label{wave2}
\end{equation}
Writing $g^{\mu\nu} = -u^\mu u^\nu + \aPi^{\mu\nu}$, we see that
$\sigma$ satisfies the following Dirichlet problem:
\begin{align}
 & \eta^{-2} \nabla_u^2 \sigma - \nabla_\mu\big(\Pi^{\mu\nu} \nabla_\nu \sigma\big)
  = \F + \G(\sigma)  &&\textrm{ in }  \D_t, \label{wave3}\\
 & \sigma = \sigma_0 &&\textrm{ on } \pa\D_t,\label{wave3bc}
\end{align}
where:
\begin{align}
 \F &= 2 (\nabla_\mu V^\nu)(\nabla_\nu V^\mu)\label{wave:fdef}\\
 \G(\sigma) &= -(\nabla_u e(\sigma) + \nabla_u  e'(\sigma)) \nabla_u \sigma.
 \label{wave:gdef}
\end{align}
The symmetric (2,0)-tensor $\eta^{-2} u^{\mu}u^\nu + \Pi^{\mu\nu}$ is known
as the acoustical metric
and \eqref{wave3} is known as the acoustic wave equation
(see \cite{Christodoulou2007}).
In this section we will derive estimates for the Dirichlet problem:
\begin{align}
  \eta^{-2}\nabla_u^2\psi - \nabla_\mu\big(\Pi^{\mu\nu}\nabla_\nu \psi\big)
   &= f, &&\textrm{ in } \D_t,\label{wave:eq}\\
  \psi &= C_0, &&\textrm{ on } \pa \D_t,\label{wave:dirich}\\
  \psi = \psi_0, \,\nabla_u \psi &= \psi_1 &&\textrm{ on } \D_0.
  \label{wave:ic}
\end{align}
where $C_0$ is a constant.

Let:
\begin{equation}
\Ew(t) = \frac{1}{2}\int_{\D_t}
\bigg((\nabla_u \psi)^2 + \Pi^{\mu\nu} \nabla_\mu\psi  \nabla_\nu\psi \bigg)
\frac{1}{|\ubar| - \utau}\dvols
+ \frac{1}{2}\int_{\D_t}  (\nabla_u \psi)^2 \big(\eta^{-2} -1\big)(-\utau)\dvols.
\end{equation}
Recall that we are writing:
\begin{align}
&\utau = u^\mu \tau_\mu, &&|\ubar|^2 = \sg(u, u).
\label{}
\end{align}
Beacuse $u$ is timelike and future-directed, $\utau < 0$ so
the first term is non-negative. In addition, $\eta \leq 1$ so the second term is
non-negative, though it degenerates if the sound speed is the same as the speed
of light.

We also note that if:
\begin{equation}
\widetilde{\E}(t) = \int_{\D_t} (\nabla_\tau \psi)^2 + \sg^{ij}\sd_i\psi
\sd_j \psi \dvols
\label{}
\end{equation}
is the energy associated to the wave equation in the metric $g$,
then:
\begin{equation}
 \widetilde{\E}(t) \leq C(\lambda,M)\, \Ew(t),
 \label{wave:coer}
\end{equation}
which follows by writing $\nabla_\tau =
\tau^\mu\nabla_\mu = -\utau\nabla_u +\aPi^\mu_\nu \tau^\nu\nabla_\mu$
and $\sd_\mu = -u_\mu \sd_u+ \aPi_\mu^\nu\sd_\nu $ and the
fact that $-(\utau + |\ubar|)$ is bounded above and below.

The basic energy estimate is:
\begin{lemma}
 Suppose that the assumptions \eqref{assump:baro}-\eqref{assump:fluid1}
 hold. If $\psi$ satisfies
 \eqref{wave:eq}-\eqref{wave:dirich}, then:
 \begin{equation}
  \Ew(t) \leq C(M,\lambda, L, t) \bigg(\Ew(0) + \int_\D | f \, \nabla_u\psi| \dvol\bigg).
  \label{wave:basic}
 \end{equation}
\end{lemma}

\begin{proof}
 We multiply \eqref{wave:eq} by
$\nabla_u \psi$ and integrate over $\D$ with respect to $\dvol$:
\begin{multline}
 \int_\D \bigg(\eta^{-2} \nabla_u^2 \psi \nabla_u \psi
 -  \nabla_\mu\big(\aPi^{\mu\nu} \nabla_\nu \psi\big) \nabla_u \psi
\bigg) \dvol\\
 = \frac{1}{2}\int_\D \eta^{-2} \nabla_u \big(\nabla_u \psi\big)^2\dvol
 + \int_{\D} \aPi^{\mu\nu} (\nabla_\nu \psi) (\nabla_\mu \nabla_u \psi)\dvol
 \\- \int_\Lambda  \aPi^{\mu\nu} \big(\nabla_\nu \psi\big)
 \big(\nabla_u \psi\big) N_\mu
 \dvol
 + \int_{\D_t}\Rr(t) \dvols - \int_{\D_0} \Rr(0)\dvols
 \label{wave:ibp}
\end{multline}
where:
\begin{equation}
 \Rr(t) = (\aPi^{\mu\nu}\tau_\mu \nabla_\nu \psi)\nabla_u \psi
 \label{wave:rdef}
\end{equation}
Beause $\psi = 0$  on $\Lambda$ and $\nabla_u$ is tangential at the boundary,
the integral over $\Lambda$ vanishes.
The first term on the right-hand side is:
\begin{multline}
 \frac{1}{2}\int_{\D}\eta^{-2} \nabla_u(\nabla_u\psi)^2 \dvol\\
 = \frac{1}{2}\int_{\D_t} \eta^{-2} (-u^\mu\tau_\mu) (\nabla_u \psi)^2
 \dvol
 - \frac{1}{2}\int_{\D_0} \eta^{-2} (-u^\mu\tau_\mu) (\nabla_u \psi)^2 \label{}
 -\frac{1}{2}\int_{\D} \nabla_\mu\big( u^\mu \eta^{-2}\big) (\nabla_u \psi)^2
 \dvol.
\end{multline}

Writing $\nabla_\mu \nabla_u \psi = \nabla_u \nabla_\mu\psi
- (\nabla_\mu u^\nu) \nabla_\nu\psi$,
the second term on the right-hand side of \eqref{wave:ibp} is:
\begin{multline}
 \int_\D \aPi^{\mu\nu} (\nabla_\nu \psi) (\nabla_u \nabla_\mu \psi)
 \dvol
 -
 \int_{\D} \aPi^{\mu\nu} (\nabla_\mu u^\alpha)(\nabla_\nu\psi)(\nabla_\alpha \psi)
  \dvol\\
 = \frac{1}{2}\int_\D \nabla_u\bigg(\aPi^{\mu\nu} (\nabla_\nu \psi)
 (\nabla_\mu \psi)\bigg)\dvol
 -\int_\D (\nabla_u \aPi^{\mu\nu})(\nabla_\mu \psi) (\nabla_\nu \psi)\dvol
 -
 \int_{\D} \aPi^{\mu\nu} (\nabla_\mu u^\alpha)(\nabla_\nu\psi)
 \nabla_\alpha \psi \dvol,
\end{multline}
and the first term here is:
\begin{multline}
 \frac{1}{2}\int_{\D}\nabla_u\bigg( \aPi^{\mu\nu} (\nabla_\mu \psi)
 (\nabla_\nu \psi)\bigg)\dvol\\
 = \frac{1}{2} \int_{\D_t} (-u^\mu \tau_\mu) \aPi^{\mu\nu} (\nabla_\mu \psi)
 (\nabla_\nu \psi) \dvols
 - \frac{1}{2} \int_{\D_t} (-u^\mu \tau_\mu) \aPi^{\mu\nu} (\nabla_\mu \psi)
 (\nabla_\nu \psi) \dvols - \int_{\D} (\nabla_{\mu'} u^{\mu'})
 \aPi^{\mu\nu} (\nabla_\mu \psi)(\nabla_\nu \psi) \dvol.
 \label{}
\end{multline}
Writing
\begin{equation}
 \e(t) = \frac{1}{2}\bigg(\eta^{-2} (\nabla_u\psi)^2+ \aPi^{\mu\nu}(\nabla_\mu
  \psi) (\nabla_\nu\psi)\bigg),
 \label{}
\end{equation}
so far we have the identity:
\begin{multline}
 \int_{\D} f \nabla_u \psi \dvol
 = \int_{\D_t}(-\utau) M(t)  + \Rr(t)
\dvols - \int_{\D_0} (-\utau) M(0)  - \Rr(0)\dvols\\
 - \frac{1}{2}\int_\D \nabla_\mu(u^\mu \eta^{-2}) (\nabla_u \psi)^2
 +  (\nabla_{\mu'}u^{\mu'} )\aPi^{\mu\nu}
 (\nabla_\mu \psi)(\nabla_\nu \psi)\dvol\\
 - \int_\D \big(\nabla_u \aPi^{\mu\nu}\big) (\nabla_\mu\psi)(\nabla_\nu\psi)
 \dvol
 +\aPi^{\mu\nu} (\nabla_\mu u^\alpha)(\nabla_\nu \psi)
 (\nabla_\alpha\psi)\dvol.
 \label{}
\end{multline}

By the assumptions \eqref{assump:baro}-\eqref{assump:ecoer}
on the equation of state and the bound \eqref{assump:fluid1},
 the terms on the
second line are bounded by $C(M, L)\Ew(t)$, and by \eqref{wave:coer},
the terms on the third line are bounded by $C(\lambda, M) \Ew(t)$.

To deal with the terms on $\D_t$, we first note that:
\begin{equation}
 \Pi^{\mu\nu} \tau_\mu\tau_\nu = g^{\mu\nu}\tau_\mu\tau_\nu
 + (\utau)^2
 = -1 + (\utau)^2 = \sg(\ubar, \ubar) = |\ubar|^2
 \label{}
\end{equation}
Therefore:
\begin{align}
 \big|\Pi^{\mu\nu}\nabla_\nu \psi \tau_\nu| |\nabla_u\psi|
 &\leq   \big(\Pi^{\mu\nu}\tau_\mu\tau_\nu \big)^{1/2}\big(\Pi^{\mu\nu}
 \nabla_\mu \psi\nabla_\nu\psi\big)^{1/2} |\nabla_u\psi|\\
 &\leq \frac{1}{2}|\bar{u}|\bigg( (\nabla_u\psi)^2 + \Pi^{\mu\nu}
 (\nabla_\mu \psi)
 (\nabla_\nu \psi)\bigg)
\end{align}
so, recalling that $\utau < 0$:
\begin{align}
 |(-u^\tau)M(t) + \Rr(t)| &\geq \frac{1}{2}\bigg|
 \big((-u^\tau) \eta^{-2} - |\bar{u}| \big)(\nabla_u \psi)^2
 + \big((-u^\tau) - |\bar{u}|\big)\aPi^{\mu\nu} \nabla_\mu\psi
 \nabla_\nu\psi\bigg)\\
 &=\frac{1}{2} \bigg|(-u^\tau - |\bar{u}|)\bigg( (\nabla_u \psi)^2
 + \aPi^{\mu\nu}\nabla_\mu\psi\nabla_\nu\psi\bigg)
 + \frac{1}{2}\big(\eta^{-2}-1\big)(-\utau) (\nabla_u \psi)^2\bigg|.
 \label{}
\end{align}
Because $\eta \leq 1$, the last term is positive. Also note that
we have:
\begin{equation}
 1 = -u^\mu u_\mu = (\utau)^2 - |\bar{u}|^2 = (-u^\tau - |\bar{u}|)
 (-u^\tau + |\bar{u}|),
 \label{}
\end{equation}
so we have shown that:
\begin{equation}
 \Ew(t) \leq \Ew(0) + C(M,\lambda, L)\int_0^t \Ew(s) \, ds,
 \label{}
\end{equation}

The estimate \eqref{wave:basic} then follows from Gr\"{o}nwall's integral inequality.
\end{proof}

To control higher derivatives, we will use the fact that $\nabla_u$ is
tangential at the boundary to control $\nabla_u^\ell \psi$, and
then the elliptic estimates from section \ref{sec:ell}
to control space derivatives.

We define:
\begin{multline}
 \Ew^r(t) = \frac{1}{2}\int_{\D_t}
 \bigg( \big(\nabla_u^{r+1} \psi\big)^2 + \Pi^{\mu\nu}
 (\nabla_\mu \nabla_u^r \psi) (\nabla_\nu \nabla_u^r \psi)
 \bigg) \frac{1}{|\ubar| - \utau} \dvols\\
 +\frac{1}{2} \int_{\D_t} \big(\nabla_u^{r+1}\psi\big)^2
(\eta^{-2} -1 ) (-\utau) \dvols.
 \label{}
\end{multline}
We will also need some
mixed norms in the interior:
\begin{align}
 ||\psi(t)||_{k,\ell} = \sum_{s \leq k, m \leq \ell}
 ||\sd^s \nabla_u^m \psi(t)||_{L^2(\D_t)}, && ||\psi(t)||_{r}
 = \sum_{k + \ell \leq r} ||\psi(t)||_{k,\ell}.
 \label{}
\end{align}
Note that $||\psi(t)||_{r, 0} + ||\psi(t)||_{r-1,1} \leq C(M,\lambda, L)
\Ew^{r-1}(t)$.
We define:
\begin{align}
 f_{k,\ell} &= \sd^k \nabla_u^\ell f,\\
 g_{k,\ell}^1(\psi)  &= \sd^{k}(\sDelta \nabla_u^\ell \psi
 - \sd_u^\ell \sDelta \psi),\\
 g_{k,\ell}^2(\psi)  &= \sd^{k}\nabla_u^\ell \nabla_\tau^2 \psi
 - \nabla_\tau^2 \sd^k\sd_u^\ell \psi,\\
  e_{k,\ell}(\psi) &= \eta^{-2}\sd^{k} \nabla_u^{\ell+2} \psi
  - \sd^{k}\nabla_u^\ell (\eta^{-2} \nabla_u^2 \psi)
 \label{}
\end{align}
as well as:
\begin{align}
 \Rr_r = \sum_{k + \ell \leq r} ||f_{k,\ell}||_{L^2(\Omega)}
 + ||g^1_{k,\ell}||_{L^2(\Omega)}
 + ||g^2_{k,\ell}||_{L^2(\Omega)} + ||e_{k,\ell}||_{L^2(\Omega)}.
 \label{}
\end{align}
We will also write:
\begin{align}
  \aDelta &= \nabla_\mu(\aPi^{\mu\nu}\nabla_\nu)\\
 \widetilde{g}_{r}(\psi) &= \aDelta \nabla_u^r \psi -
\nabla_u^\ell \aDelta \psi
 \label{gtilde}
\end{align}

Because $\nabla_u$ is tangential at the boundary,
$\psi_{r} = \nabla_u^r \psi$ satisfies the wave equation:
\begin{align}
 \eta^{-2}\nabla_u^2 \psi_r - \nabla_{\mu}\big(\aPi^{\mu\nu} \nabla_\nu
 \psi_r\big)
 &= f_{0,r} + \widetilde{g}_{0,r}(\psi) + e_{0,r}(\psi) &&\textrm{ in }  \Omega\\
 \psi_r &= 0 &&\textrm{ on } \pa\Omega,
\end{align}
and so by \eqref{wave:basic}, we have:
\begin{multline}
 \Ew^{r}(t) \leq C\bigg( \Ew^{r}(0) + \int_{0}^t
 \big( ||f_{0,r}(s)||_{L^2(\Omega)} +
 ||\widetilde{g}_{0,r}(s)||_{L^2(\Omega)} +
  ||e_{0,r}(s)||_{L^2(\Omega)} \big)
 ||\nabla_u \psi_\ell(s)||_{L^2(\Omega)}\,ds\bigg),
 \label{wave:du}
\end{multline}
with $C = C(M, \lambda,L, t)$.

While it was convenient to use the decomposition $g^{\mu\nu} = -u^\mu u^\nu
+ \aPi^{\mu\nu}$ to control the material derivatives $\nabla_u$,
to control the spatial derivatives it is simpler
to
use the decomposition $g^{\mu\nu} = -\tau^\mu \tau^\nu + \sg^{\mu\nu}$.
With $\sDelta$ the Laplace-Beltrami operator on $\D_t$ (defined
in \eqref{assump:lb})
we re-write \eqref{wave2} as an elliptic equation:
\begin{equation}
  \sDelta \sigma = (\eta^{-2}-1)\nabla_u^2\sigma + \nabla_\mu(\tau^{\mu}
  \tau^\nu\nabla_\nu \sigma) + \F + \G(\sigma).
 \label{}
\end{equation}

We will now prove estimates for the solution to the Dirichlet
problem:
\begin{align}
 &\sDelta \psi = (\eta^{-2} - 1)\nabla_u^2\psi + \nabla_\tau^2 \psi
 + f &&\textrm{ in } \Omega \label{ellform1}\\
 &\psi = C_0 &&\textrm{ on } \pa\Omega \label{ellform2},
\end{align}
where $C_0$ is a constant.
The term $\nabla_u^2 \psi$ is lower order because it vanishes on the
boundary.
The term $\nabla_\tau^2 \psi$ does not vanish on the boundary and is above
top order, but we can use \eqref{assump:stat} to control it in terms
of $\nabla_u^2\psi$ and a small parameter times $\sd^2\psi$.

Then the elliptic estimates from section \ref{sec:ell} imply:
\begin{lemma}
 Under the above hypotheses, if $k + \ell = r$
 with $r \geq 5$:
 \begin{multline}
  ||\psi(t)||_{k,\ell} \leq C \bigg( \Ew^r(t)
  + \Rr_{r-1}(t)\\
  + \big(||\theta(t)||_{H^{r-2}(\pa\D_t)} + ||u(t)||_{r-1}+
  P(M, ||u(t)||_{r-2})\big)\bigg) ||\psi(t)||_{r-1}
  \label{wave:mixbd}
 \end{multline}
 where $C = C(M,\lambda, L, K, R, Vol(\D_t), ||\theta||_{H^2(\pa \D_t)})$.
\end{lemma}
\begin{proof}
The cases $k = 0,1$ follow from \eqref{wave:du}.
For $k \geq 2$ we use the elliptic estimate \eqref{ellbdyfn2}
and then\eqref{elllotbdy} to control the term involving
projected derivatives:
\begin{multline}
 ||\sd^k \nabla_u^\ell \psi||_{L^2(\D_t)}
 \leq C\bigg( ||(\nas^{k-2} \theta) (\sd_\N \nabla_u^\ell \psi)||_{L^2(\pa \D_t)}\\
 + (||\theta||_{L^\infty(\pa \D_t)} + R) ||\sd^{k-2} \sDelta \nabla_u^\ell
 \psi||_{L^2(\D_t)}
 + \widetilde{C}
  ||\sDelta \nabla_u^\ell \psi||_{H^{k-3}(\D_t)}\bigg),
 \label{wave:mixstart}
\end{multline}
where $\widetilde{C} = \widetilde{C}(K, R, Vol(\D_t), ||\theta||_{L^\infty(\pa\D_t)},
||\theta||_{H^{k-3}(\pa \D_t)})$.

 When $k \leq r-2$, we bound the first term
$||\nas^{k-2} \theta||_{L^\infty(\pa \D_t)}
||\sd_\N \nabla_u^\ell \psi||_{L^2(\pa \D_t)}$, and by Sobolev embedding
\eqref{bdysob2} and the estimate \eqref{ellbdy1} the result is bounded
by the right-hand side of \eqref{wave:mixbd}.
For
 $k = r-1, r-2$, we bound this by $||\theta||_{H^{r-2}(\pa \D_t)}
 ||\nabla \nabla_u^\ell \psi||_{L^\infty(\pa \D_t)}
 \leq C||\theta||_{H^{r-2}(\pa \D_t)}
 ||\nabla \nabla_u^\ell\psi||_{H^2(\pa \D_t)}$. Again by
 \eqref{ellbdy1}, this second factor is bounded by the right-hand
 side of \eqref{wave:mixbd} because $r \geq 5$.

To control $||\sd^{k-2}\sDelta \nabla_u^\ell \psi||_{L^2(\D_t)}$,
we use \eqref{ellform1} and \eqref{dtfn}:
\begin{multline}
 ||\sd^{k-2} \sDelta \nabla_u^\ell \psi||_{L^2(\D_t)}
 \leq C(L)||\sd^{k-2} \nabla_u^{\ell+2}\psi||_{L^2(\D_t)}
 + \lambda^2 ||\sd^k \nabla_u^\ell \psi||_{L^2(\D_t)}
 + \Rr_{r-1}(t)\\
 + ||u||_{r-1}||\psi||_1 + P(M, ||u||_{r-2})||\psi||_{r-1}.
 \label{}
\end{multline}
For sufficiently small $\lambda$, returning to \eqref{wave:mixstart}
we have:

 \begin{multline}
   ||\sd^{k}\nabla_u^\ell \psi||_{L^2(\D_t)}
 \leq C \bigg( ||\sd^{k-2}\nabla_u^{\ell+2} \psi||_{L^2(\D_t)}
  + \Rr_{r-1}(t)\\
  + \big(||\theta(t)||_{H^{r-2}(\pa\D_t)} + ||u(t)||_{r-1}+
  P(M, ||u(t)||_{r-2})\big)\bigg) ||\psi(t)||_{r-1},
  \label{wave:mixbd2}
\end{multline}
with $C$ as in \eqref{wave:mixbd}.
Replacing $k$ with $k-2$, $\ell$ with $\ell+2$ and repeating
this argument gives \eqref{wave:mixbd}.
\end{proof}

We finally return to the equation \eqref{wave3} satisfied by the enthalpy
$\sigma$. We define:
\begin{multline}
 \EW^r(t) = \sum_{\ell \leq r} \int_{\D_t}
 \bigg( (\nabla_u^{r+1} \sigma(t)\big)^2 + \aPi^{\mu\nu} (\nabla_\mu
 \nabla_u^r\sigma(t))
 (\nabla_\nu \nabla_u^r\sigma(t))\bigg) \frac{1}{|\ubar| - \utau} \dvols
 \\
 + \frac{1}{2} \int_{\D_t} (\nabla_u^{r+1} \sigma)^2
(\eta^{-2} - 1) (-\utau) \dvols.
 \label{}
\end{multline}

\begin{cor}
 If $\sigma$ satifies \eqref{wave} and the assumptions
 \eqref{assump:rm}, \eqref{assump:stat}, \eqref{assump:baro},
 \eqref{assump:sound},\eqref{assump:fluid1} and \eqref{assump:fluid2}
 hold, then there is a polynomial $P$ so that:
 \begin{multline}
  \EW^r(t) \leq C\bigg( \EW^r(0) + \int_0^t \bigg(
  ||\nabla^{r+1}\sigma(s)||_{L^2(\D_t)} \\+ P(||\sigma(s)||_{r}, ||V(s)||_{r},
  M, L, R) \bigg)
   ||\nabla_u^{r+1}\sigma(s)||_{L^2(\D_t)} \, ds\bigg),
  \label{sigmawavebd}
 \end{multline}
 and, for $k + \ell = r+1$:
 \begin{equation}
  ||\sigma||_{k,\ell} \leq C\bigg( \EW^r(t) + P(||\sigma||_r,
  ||V||_r,R)\bigg),
  \label{sigmaellbd}
 \end{equation}
 where $C = C(t, M,\lambda, K, ||\theta||_{H^{r-2}(\pa \D_t)}, Vol(\D_t),
 L)$.
\end{cor}

\begin{proof}
We write:
\begin{align}
 \F_{k,\ell} = \sd^k \nabla_u^\ell \F,\\
 \G_{k,\ell}(\sigma) = \sd^k \nabla_u^\ell \G(\sigma),
 \label{}
\end{align}
where $\F, \G$ are defined in \eqref{wave:fdef}-\eqref{wave:gdef}.

By Lemma \ref{remainderlem} in the appendix we have
the following estimates when $k + \ell = r$:
\begin{align*}
 ||\F_{k,\ell}||_{L^2(\D_t)}
 &\leq C(M, \lambda) \bigg( ||\nabla \sd^k \nabla_u^\ell V||_{L^2(\D_t)} + P(||V||_r,
 ||\sigma||_r, R)\bigg),\\
 ||g_{k,\ell}^1(\sigma)||_{L^2(\D_t)} + ||g_{k,\ell}^2(\sigma)||_{L^2(\D_t)}
 + ||\widetilde{g}_r(\sigma)||_{L^2(\D_t)}
 &\leq C(M,\lambda) \bigg(
 ||\nabla^{k+ \ell+1} u||_{L^2(\D_t)} + P(||u||_r)
 ||\sigma||_r\bigg),\\
 ||e_{k,\ell}(\sigma)||_{L^2(\D_t)}
 + ||\G_{k,\ell}(\sigma)||_{L^2(\D_t)} &\leq C(M,\lambda,L) \bigg(
 ||\sd^k \nabla_u^\ell \sigma||_{L^2(\D_t)} + P(||u||_{r-1}, ||\sigma||_{r-1})
 \bigg),
 \label{}
\end{align*}
for some polynomials $P$.

Then \eqref{sigmawavebd} and \eqref{sigmaellbd} follow from
\eqref{wave:du} and \eqref{wave:mixbd}.

\end{proof}

\section{The energy identity}

We let $\beta_{\mu} = \beta_{I\mu}$ and $\alpha_\mu = \alpha_{I \mu}$.
As in previous sections, we will write $\sbe_\mu = \sPi^\nu_\mu \beta^\nu$
as well as $\div \beta = g^{\mu\nu}\nabla_\mu \beta_{I\nu}$. We will also
write:
\begin{align}
 &\sdiv \alpha = \sg^{\mu\nu} \sd_{\mu} \alpha_{I \nu},
 &&\scurl \alpha_{\mu\nu} =  \sd_{\mu} \alpha_{I \nu} - \sd_\nu \alpha_{I\mu},
 \label{}
\end{align}
We let $\Pis$ be the projection to $\pa \D_t$ extended to the interior,
defined in \eqref{extPidef}.
If $\omega$ is a $(0, r)$ tensor we set:
\begin{equation}
 Q(\omega, \omega) = \sg^{IJ} \big(\Pis_I^K \omega_K\big)
 \big(\Pis_J^L \omega_L\big)
 = \sg^{i_1 j_1}\cdots \sg^{i_r j_r}
 \big(\Pis^{k_1}_{i_1}\cdots \Pis^{k_r}_{i_r}\omega_{k_1\cdots k_r}\big)
 \big(\Pis^{\ell}_{j_1}\cdots \Pis^{\ell_r}_{j_r}
 \omega_{\ell_1\cdots \ell_r}\big)
 \label{}
\end{equation}
and then we define:
\begin{multline}
 E(t)
 = \frac{1}{2}\int_{\D_t}\bigg(
 \big(\tau^\mu\tau^\nu + \sg^{\mu\nu}\big)
 \gamma^{ij}Q(\sd_i \beta_\mu, \sd_j\beta_\nu)
 \bigg)\frac{\sqrt{\sigma}}{|\ubar| - \utau}\dvols\\
 + \frac{1}{2}\int_{\D_t}\gamma^{ij} Q(\sd_i\alpha, \sd_j \alpha)(-\Vtau) \kappa \dvols
 \\+ \frac{1}{2}
 \int_{\pa \D_t} \gamma^{ij} Q(\sd_i \alpha, \sd_j \alpha)  (-\Vtau) \nu\dvolbdy,
 \label{}
\end{multline}
where $0 < \nu,\kappa < \infty$. Suppose that the following bounds for
$\nu,\kappa$ and $\pa \D_t$ hold:
\begin{align}
 &\left|\frac{\nabla_V \kappa}{\kappa}\right| \leq C_1,&&\textrm{ in }
\D_t,\label{kappa}\\
 &\left|\frac{\nabla_V \nu}{\nu}\right| \leq C_2,
 &&\textrm{ on } \pa \D_t\label{nu}\\
 &|\theta| + \frac{1}{\iota_0} \leq K,
 &&\textrm{ on } \pa \D_t.
 \label{K}
\end{align}
Then we have the following energy estimate:
\begin{prop}
  If the assumptions \eqref{kappa}-\eqref{K} and \eqref{assump:rm}
  hold, then:
 \begin{align}
  E(t) -E(0) &\leq
    \int_0^t \sqrt{E(s)}\bigg( ||\nabla_V \sd \beta + \nabla
  \sd\alpha||_{L^2(\D_s)} + ||\div \sd \beta + \kappa \nabla_V
   \sd\alpha||_{L^2(\D_s)}\\
  &+ ||\Pis\big(N^\mu\sd \beta_\mu - \nu \nabla_V \sd \alpha\big)||_{L^2(\pa \D_s)}
  + || \sd \alpha + V^\mu \sd \beta_\mu||_{L^2(\D_s)}\bigg)\\
  &+ C(R)\bigg( ( C_1 + C_2 + K)E(s) + (||\sdiv \alpha||_{L^2(\D_s)}^2
  + ||\scurl \alpha||_{L^2(\D_s)}^2\\
  & + ||\sdiv \beta||_{L^2(\D_s)}^2
  + ||\scurl \beta||_{L^2(\D_s)}^2\bigg) \,ds.
  \label{en:ident}
 \end{align}
\end{prop}
\begin{remark}
  We will apply the above lemma with $\beta = \sd^{r-1} V$,
  $\alpha = \frac{1}{2}\sd^{r-1} \sigma$,
  $\kappa = 2e'(\sigma)$ and $\nu = \frac{2}{|\nabla_N \sigma|}$.
  The first and second
  terms on the right-hand side then vanish to highest order by \eqref{int:dtv}
  and \eqref{int:dtesig}.
  The third term is a tangential projection
  $\Pis$ of:
  \begin{equation}
   N^\mu \sd^r V_\mu - \frac{1}{|\nabla_N \sigma|} \nabla_V \nabla^r \sigma
   = -\frac{1}{|\nabla_N \sigma|}
   \bigg( \nabla^\mu \sigma  \nabla^r V_\mu + \nabla_V \nabla^r \sigma\bigg).
   \label{}
  \end{equation}
  To highest order, $\nabla^\mu \sigma \nabla^r V_\mu + \nabla_V \nabla^r
  \sigma = \nabla^r \nabla_V \sigma$, and because $\nabla_V \sigma = 0$,
  $\Pis \nabla^r \nabla_V \sigma$ is lower order.
  The fourth term will be lower order because $\sigma = -V^\mu V_\mu$.
\end{remark}
\begin{proof}

  We multiply $\nabla_V \sd_i \beta_{I\mu} + \nabla_\mu \sd_i \alpha_I$
  by $\gamma^{ij} \gamma^{IJ} g^{\mu\nu} \sd_j \beta_{J\nu}$ and integrate
  over $\D$ to get:
\begin{align}
 \int_{\D} \gamma^{ij} &\gamma^{IJ} g^{\mu\nu}\bigg((\nabla_V \sd_i \beta_{I\mu})
 (\sd_j\beta_{J\nu}) + (\nabla_\mu\sd_i \alpha_{I} )(\sd_j \beta_{J\nu})\bigg)
 \dvol\\
 &=  \int_{\D} \gamma^{ij} \gamma^{IJ}g^{\mu\nu}\bigg(\frac{1}{2}
 \nabla_V\big( (\sd_i \beta_{I\mu}) (\sd_j \beta_{J\nu})\big)
 -  (\sd_i \alpha_I) (\nabla_\mu \sd_j \beta_{J\nu})\bigg)\dvol
 \\
 &- \int_{\D_t}\gamma^{ij} \gamma^{IJ} (\sd_i \alpha_I)( \sd_j \beta_{I\mu}
  \tau^\mu) \dvols
 + \int_{\D_0}\gamma^{ij} \gamma^{IJ} (\sd_i \alpha_J)( \sd_j \beta_{I\mu}
 \tau^\mu) \dvols\\
 &+ \int_{\Lambda}\gamma^{ij} \gamma^{IJ} (\sd_i \alpha_I )
 (\sd_j \beta_{J\mu})N_\mu \dvolbdy
 - \int_{\D} \nabla_\mu(\gamma^{ij}\gamma^{IJ})g^{\mu\nu} (\sd_i\alpha_I)(\sd_j
 \beta_{J\nu})\dvol
\label{en:int}\end{align}

The first term on the second line is:
\begin{multline}
 \frac{1}{2}\int_{\D_t} \gamma^{ij}g^{\mu\nu}Q(\sd_i \beta_\mu, \sd_j \beta_\nu)
 (-V^\tau)\dvols
 -\frac{1}{2}\int_{\D_0} \gamma^{ij}g^{\mu\nu}Q(\sd_i \beta_\mu, \sd_j \beta_\nu)
  (-V^\tau)\dvols
 \\ \frac{1}{2} \int_{\Lambda}\gamma^{ij}g^{\mu\nu}Q(\sd_i \beta_\mu, \sd_j \beta_\nu)
 V^{\mu'} N_{\mu'}
 -
  \frac{1}{2}\int_{\D} \nabla_{\mu'}\big(\gamma^{ij}\gamma^{IJ} V^{\mu'}\big)
  g^{\mu\nu}(\sd_i \beta_{I\mu})( \sd_j \beta_{J\nu}),
\end{multline}
and the term on $\Lambda$ vanishes by the boundary condition \eqref{int:free}.

The second term on the second line of \eqref{en:int} is:
\begin{multline}
 \int_{\D} \gamma^{ij}\gamma^{IJ}(\sd_i\alpha_I)
 (\nabla_V \sd_j\alpha_J)\kappa \dvol
 - \int_{\D}\gamma^{ij}\gamma^{IJ}(\sd_i\alpha_I)(\div \sd_j\beta_J +
  \kappa \nabla_V\sd_j \alpha_J)\dvol\\
 = \frac{1}{2}\int_{\D_t} \gamma^{ij}\gamma^{IJ}
 (\sd_i\alpha_I)(\sd_j\alpha_J)\kappa (-V^\tau) \dvols
 - \frac{1}{2}\int_{\D_0} \gamma^{ij}\gamma^{IJ}
 (\sd_i\alpha_I)(\sd_j\alpha_J)\kappa (-V^\tau)\dvols\\
 - \int_{\D} \nabla_{\mu}\big(\gamma^{ij}\gamma^{IJ}\kappa V^\mu\big)
 (\sd_i\alpha_I)(\sd_j\alpha_J)
 - \gamma^{ij}\gamma^{IJ}(\sd_i\alpha_I)(\div \sd_j\beta_J + \kappa \nabla_V
 \sd_j \alpha_J)\dvol.
 \label{}
\end{multline}

The integral over $\Lambda$ from \eqref{en:int} is:
\begin{multline}
  \int_{\Lambda}\gamma^{ij}Q(\sd_i\alpha, \sd_j\beta_\mu)N^\mu \dvolbdy
  \\
 = \int_{\Lambda} \gamma^{ij}Q (\sd_i \alpha,
   \nabla_V \sd_j \alpha)\nu\dvolbdy
   + \int_{\Lambda}\gamma^{ij}Q(\sd_i \alpha,
   (\sd_j\beta_\mu) N^\mu - \nu \nabla_V \sd \alpha)\dvolbdy.
 \label{}
\end{multline}
By the boundary condition \eqref{int:free},
$\nabla_V = V^\mu \nabla_\mu = V^\mu \nas_\mu$, where $\nas$ is the intrinsic
covariant derivative on $\Lambda$. Since the boundary of $\Lambda$ is
$\pa \D_t \cup \pa \D_0$, applying the divergence theorem on $\Lambda$
gives:
\begin{multline}
 \int_{\Lambda} \gamma^{ij}\gamma^{IJ} (\sd_i \alpha_I)
   (\nabla_V \sd_j \alpha_J)\nu\dvolbdy\\
   =\frac{1}{2} \int_{\pa\D_t} \gamma^{ij} \gamma^{IJ} (\sd_i\alpha_I)(\sd_j \alpha_J)
  (-V^\mu\tau_\mu) \nu\dvolbdy
   - \frac{1}{2}\int_{\pa\D_0} \gamma^{ij} \gamma^{IJ} (\sd_i\alpha_I)(\sd_j \alpha_J)
   (-V^\mu\tau_\mu) \nu\dvolbdy\\
   -\int_{\Lambda} \nabla_{\mu}\big(\gamma^{ij}\gamma^{IJ} V^\mu \nu\big)
  (\sd_i\alpha_I)(\sd_j \alpha_J) \dvolbdy.
 \label{en:bdyterm}
\end{multline}

If we write
$\sd_i \alpha = -V^\nu \sd_i\beta_\nu + (\sd_i\alpha +
V^\nu\sd_i \beta_\nu)$ and $g^{\mu\nu} = -\tau^\mu \tau^\nu + \sg^{\mu\nu}$,
then the integrals over $\D_t$ from the above calculation are:
\begin{multline*}
 \int_{\D_t} \frac{1}{2}\gamma^{ij}\bigg(g^{\mu\nu}
 Q(\sd_i \beta_\mu, \sd_j \beta_\nu)
  + Q(\sd_i \alpha, \sd_j \alpha)\kappa
\bigg) (- V^\tau) -  Q(\sd_i \alpha, \sd_j \beta_{\mu})\tau^\mu
  \dvols\\
  =
  \int_{\D_t} \frac{1}{2}\gamma^{ij}\bigg( \big(-\tau^\mu\tau^\nu
  Q(\sd_i \beta_\mu,
   \sd_j \beta_\nu) + \sg^{\mu\nu}
   Q(\sd_i \beta_\mu, \sd_j \beta_\nu)
   + Q(\sd_i\alpha, \sd_j \alpha)\kappa \big)(-V^\tau)\\
   + Q(\sd_i\beta_\nu,\sd_j \beta_\mu) V^\nu \tau^\mu
   - Q(\sd_i\alpha + V^\nu \sd_i\beta_\nu, \sd_j \beta_\mu)\tau^\mu
   \bigg)\dvols
 \label{}
\end{multline*}
Next we write:
\begin{equation}
 -Q(\sd_i \beta_\nu, \sd_j \beta_\mu) V^\nu \tau^\mu
 =  Q(\sd_i\beta_\nu, \sd_j \beta_\mu) \tau^\nu \tau^\mu V^\tau
 - Q(\sd_i\beta_\nu, \sd_j \beta_\mu) \sV^\nu \tau^\mu,
 \label{}
\end{equation}
and so the above is:
\begin{multline}
 \frac{1}{2}\int_{\D_t} \gamma^{ij} \bigg( \tau^\mu\tau^\nu Q(\sd_i \beta_\mu,
 \sd_j \beta_\nu) + \sg^{\mu\nu}
 Q(\sd_i \beta_\mu, \sd_j \beta_\nu)
 + Q(\sd_i\alpha, \sd_j \alpha)\kappa\bigg)(-V^\tau)\dvols\\
 + \int_{\D_t} \gamma^{ij}Q(\sd_i \beta_\nu, \sd_j \beta_\mu) \sV^\nu \tau^\mu
 \dvols
 +  \int_{\D_t}\gamma^{ij}Q(\sd_i\alpha + V^\nu \sd_i\beta_\nu, \sd_j \beta_\mu)
 \tau^\mu
 \dvols.
 \label{en:simp}
\end{multline}

Now we note that:
\begin{equation}
 |Q(\sd_i \beta_\nu, \sd_j \beta_\mu) \sV^\nu \tau^\mu|
 \leq \frac{1}{2} |\sV| \bigg( \tau^\mu\tau^\nu Q(\sd_i \beta_\mu,
 \sd_j \beta_\nu) + \sg^{\mu\nu} Q(\sd_i \beta_\mu, \sd_j \beta_\nu)\bigg),
 \label{}
\end{equation}
so the sum of the terms in the first line of \eqref{en:simp} with the first term
on the second line is bounded below by:
\begin{multline}
  \frac{1}{2}\int_{\D_t} \gamma^{ij}\bigg(\tau^\mu\tau^\nu Q(\sd_i\beta_\mu, \sd_j
  \beta_\nu) (-V^\tau - |\sV|) + \sg^{\mu\nu}
  Q(\sd_i \beta_\mu, \sd_j \beta_\nu) (-V^\tau - |\sV|)
  + Q(\sd_i \alpha, \sd_j \alpha) \kappa (-V^\tau) \dvols\\
  = \frac{1}{2} \int_{\D_t} \gamma^{ij}\bigg(
  \tau^\mu \tau^\nu Q(\sd_i\beta_\mu, \sd_j \beta_\nu) + \sg^{\mu\nu}Q(\sd_i
  \beta_\mu,\sd_j \beta_\nu) \bigg) \frac{\sigma}{|\sV| - V^\tau} \dvols
  + \int_{\D_t} \gamma^{ij} Q(\sd_i \alpha, \sd_j \alpha)\kappa(-V^\tau) \dvols,
 \label{}
\end{multline}
where we have used that $(-V^\tau - |\sV|)(-V^\tau + |\sV|) = \sigma$.
This plus the first term from \eqref{en:bdyterm} is $\E(t)$.

The remaining terms involving integrals over $\D$ and $\Lambda$ are
the error terms:
\begin{multline}
\int_\D \gamma^{ij} \bigg(g^{\mu\nu}Q(\nabla_V \sd_i \beta_\mu
+ \nabla_\mu \sd_i \alpha, \sd_j \beta_\nu)
- Q( \sd_i \alpha, \div \sd_j \beta + \kappa \nabla_V \sd_j \alpha)
- Q(\sd_i\alpha + V^\nu\sd_i\beta_\nu, \sd_j \beta_\mu)\tau^\mu\bigg)
\dvol\\
+ \int_{\Lambda} Q(\sd_i\alpha, N^\mu\sd_j\beta_\mu - \nu \nabla_V \sd_j
\alpha) \dvolbdy,
\label{en:lot}
\end{multline}
and the terms where derivatives fall on the coefficients:
\begin{multline}
 \int_{\D} \nabla_{\mu}( \gamma^{ij} \gamma^{IJ}) g^{\mu\nu} (\sd_i\alpha_I)
 (\sd_j\beta_{J\nu})
 + \nabla_{\mu'}(\gamma^{ij}\gamma^{IJ}V^{\mu'}) g^{\mu\nu}
 \sd_i \beta_{I\mu}\sd_j\beta_{J\nu}
  + \nabla_\mu (\gamma^{ij}\gamma^{IJ} \kappa V^\mu)
 (\sd_i\alpha_I)(\sd_j\alpha_J)\dvolbdy\\
 + \int_{\Lambda} \nabla_\mu(\gamma^{ij}\gamma^{IJ} V^\mu (-\nu))
 (\sd_i \alpha_I)(\sd_j\alpha_J) \dvol
 \label{en:dproj}
\end{multline}

The terms in \eqref{en:lot} are bounded by the right-hand side of
\eqref{en:ident}.
To deal with the first term in \eqref{en:dproj}, we first write:
\begin{equation}
 \sd_j\beta_{J\nu} = -(\sd_j\beta_{J\tau})\tau_\nu +
 \sd_j\sbe_{J\nu} - \beta_{J\nu'} \sd_j\tau^{\nu'}\tau_\nu,
 \label{}
\end{equation}
where $\beta_{J\tau} = \beta_{J\mu}\tau^\mu$.
To control the terms resulting from the components $\sbe$, we use
the elliptic estimate \eqref{ellpw}:
\begin{multline}
 \bigg|\int_{\D_t} \nabla_\mu(\gamma^{ij} \gamma^{IJ})
 g^{\mu} (\sd_i \alpha_I)(\sd_j \sbe_{J\nu'})\dvols\bigg|\\
 \leq C\bigg(KE + \big(||\sdiv \alpha||_{L^2(\D_t)} +
 ||\scurl \alpha||_{L^2(\D_t)}\big)\big( ||\sdiv \sbe||_{L^2(\D_t)}
 + ||\scurl \sbe||_{L^2(\D_t)}\bigg).
 \label{}
\end{multline}
We will deal with the term coming from $\sd_j \beta_{J\tau}$
momentarily.

To deal with the second term in \eqref{en:dproj}, we use the same
argument to bound the spatial components $\sbe$ by:
\begin{multline}
 \bigg|\int_{\D_t} \nabla_{\mu'}(\gamma^{ij} \gamma^{IJ}V^{\mu'})
 g^{\mu} (\sd_i \sbe_{I\mu})(\sd_j \sbe_{J\nu})\dvols\bigg|\\
 \leq C\bigg(KE + \big(||\sdiv \sbe||_{L^2(\D_t)} +
 ||\scurl \sbe||_{L^2(\D_t)}\bigg)^2.
\end{multline}

To control the second term in \eqref{en:dproj}, we
use the same argument as well as the assumption \eqref{kappa} on $\kappa$:
\begin{multline}
 \bigg|\int_{\D_t} \nabla_\mu(\gamma^{ij}\gamma^{IJ} \kappa V^\mu)
 (\sd_i \alpha_I)(\sd_j\alpha_J) \dvols\bigg|\\
 \leq C\bigg( (C_1 + K)E + ||\sdiv \alpha||_{L^2(\D_t)}^2 +
  ||\scurl\, \alpha||_{L^2(\D_t)}^2\bigg).
 \label{}
\end{multline}

For the boundary term in \eqref{en:dproj}, we instead use
the assumption \eqref{nu} on $\nu$:
\begin{equation}
 \bigg|\int_{\D_t} \nabla_\mu(\gamma^{ij}\gamma^{IJ} V^\mu (-\nu))
 (\sd_i \alpha_I)(\sd_j \alpha_J) \dvolbdy\bigg|
 \leq C(C_2 + 1)E.
 \label{}
\end{equation}

Finally, to deal with the terms coming from $\sd_j \beta_{J\tau}$, we write:
\begin{equation}
 \nabla_i\beta^\tau
 = \frac{1}{2V^\tau}\bigg(\sd_i\alpha + 2\sd_i\beta_\mu \sV^\mu\bigg)
 + \bigg(\sd_i\beta^\tau - \frac{1}{2V^\tau}(\sd_i\alpha +
  2\sd_i\beta_\mu \sV^\mu)\bigg).
\end{equation}
We can then use the same arguments as above to bound the terms resulting
from the terms in the first bracket here, and after multiplying
by $V^\tau$, the terms
in the second bracket here are bounded by the fourth term on the right-hand side
of \eqref{en:ident}.
\end{proof}

\section{The higher-order equations}
\begin{lemma}
 Suppose that $u$ satisfies \eqref{assump:stat}
 and that the assumptions \eqref{assump:rm}, \eqref{assump:baro}, \eqref{assump:sound},
 \eqref{assump:lower} and \eqref{assump:ecoer} hold. If $V, \sigma$ satisfy
 \eqref{int:dtv} and \eqref{int:dtesig} then:
 \begin{align}
  &|\nabla_V \sd^r V + \tfrac{1}{2}\nabla \sd^r \sigma|
  + |\nabla_V \curl \sd^{r-1} V|
  \leq C \sum_{s=1}^r |(\sd^{s} V)(\sd^{r-s} \nabla V)|
   + R\sum_{s = 1}^{r-1} |\nabla^s V| , \label{he:mom}\\
  &|\tfrac{1}{2}\sd^r \sigma + V^\mu \sd^r V_\mu|
  \leq C \sum_{s = 1}^r |\sd^{s}V^\mu|
  |\sd^{r-s}V_\mu|,\label{he:norm}\\
  &|\Pis(\nabla_V \sd^{r}\sigma + (\sd^r V)\cdot \nabla \sigma
  - \sd^r \nabla_V \sigma) | \leq  C\sum_{s = 1}^{r-1} \big|
  \Pis \big( (\sd^s V^\mu)(\sd^{r-s}\nabla_\mu\sigma)\big)\big|.
  \label{he:bdy}
\end{align}
and there is a polynomial $P$ so that:
\begin{multline}
  |\div \sd^r V + e'(\sigma) \nabla_V \sd^r \sigma| \\
  \leq  C(M)
    \sum_{s=1}^r |\sd^{s} V| |\sd^{r-s} \nabla V|
    + R\sum_{s = 0}^{r-1} |\nabla^s V|
  +  P(L, |\sd\sigma|, ..., |\sd^{r-1} \sigma|).\label{he:mass}
 \end{multline}

 For $k + \ell = r-1$, $k \geq 2$, there is a polynomial $P$
 so that:
 \begin{multline}
  |\sd^{k-2} \sDelta \nabla_u^\ell \sigma
  - (\sd^{k-2} \nabla_u^{\ell-1} \sDelta u^\mu)\nabla_\mu \sigma|\\
  \leq C\bigg( |\sd^{k-2} \nabla_u^{\ell+2}\sigma|
  + \lambda |\sd^k \nabla_u^\ell \sigma|
  + \sum_{s = 0}^{r-1} |(\nabla^{s+1}V) (\nabla^{r-s+1} V)|\\
  +
  P(R, L,  |\nabla \sigma|,...|\nabla^{r-2} \sigma|, |\nabla u|,...,
  |\nabla^{r-2} u|)\bigg).
  \label{he:lapl}
 \end{multline}
\end{lemma}
\begin{proof}
  By \eqref{int:dtv}, $(\curl \nabla_V V)_{\mu\nu} = 0$. Therefore,
  writing $\nabla_{\mu} (V^{\nu'}\nabla_{\nu'})
  = (\nabla_\mu V^{\nu'} )\nabla_{\nu'} + V^{\nu'}$, we have:
  \begin{equation}
   (\nabla_V \curl V)_{\mu\nu} = \nabla_\mu V^{\nu'} \nabla_{\nu'} V_\nu
   - \nabla_\nu V^{\nu'}\nabla_{\nu'} V_\mu.
   \label{}
  \end{equation}
  The estimates \eqref{he:mom} and \eqref{he:bdy} then
  follow from the product rule
  and the bound \eqref{assump:rm} for the Riemann tensor. For
  \eqref{he:norm} we additionally use that $\sigma
  = -V^\mu V_\mu$.

  To get \eqref{he:mass}, we can argue as in the proof of
  \eqref{crest} to get:
  \begin{equation}
   |\nabla^r (e'(\sigma) \nabla_V \sigma) - e'(\sigma)
   \nabla^r \nabla_V \sigma|
   \leq P(L, |\sd \sigma|,..., |\sd^{r-1} \sigma|)
   \label{}
  \end{equation}
  The estimate \eqref{he:lapl} is similar and uses the equation
  \eqref{ellform1} as well as \eqref{dtfn}.
\end{proof}

We will also need estimates for $\sdiv V$ and $\scurl V$ in terms
of $\div V$ and $\curl V$:
\begin{lemma}
  Suppose that \eqref{assump:stat} holds and that $V, \sigma$ satisfy
  \eqref{int:dtv} and \eqref{int:dtesig}. Then:

  \begin{align}
   |\sdiv \sd^{r-1} V| \leq |\sd^{r-1} \nabla_V e(\sigma)| +
   C(M,L)\bigg(|\nabla_V \sd^{r-1} V| +
   \lambda|\sd^{r}V|  + P(|\nabla^{r-2} V|, R)\sum_{k = 1}^{r}|\nabla^{k-1}V|
   \bigg), \label{he:sdiv}\\
   |\scurl \sd^{r-1} V| \leq |\sd^{r-1}  \curl V|
   + C(M,L)\bigg(|\nabla_V \sd^{r-1} V| +
   \lambda|\sd^{r}V|  + P(|\nabla^{r-2} V|, R)\sum_{k = 1}^{r}|\nabla^{k-1}V|
   \bigg).
   \label{he:scurl}
  \end{align}
\end{lemma}
\begin{proof}
  By the definition \eqref{sddef}, we have:
  \begin{equation}
    \nabla_\mu V^\nu = \sd_\mu V^\nu + \tau_\mu \tau_\nu
    \tau^{\mu'}\tau^{\nu'}\nabla_{\mu'} V^{\nu'}
    -(\tau_\mu \tau^{\mu'} \sPi^{\nu'}_{\nu}
    +\tau_\nu \tau^{\nu'} \sPi^{\mu'}_{\mu} )\nabla_{\mu'} V^{\nu'},
  \end{equation}
  and taking the trace and anti-symmetric part of this identity shows that:
  \begin{align}
   \sdiv V &= \div V + \tau_\mu \nabla_\tau V^\mu,\label{sdivident}\\
   \scurl V_{\mu\nu} &= \curl V_{\mu\nu}
   + (AV)_{\mu\nu} + (BV)_{\mu\nu},
   \label{}
  \end{align}
  where $AV$ is the anti-symmetric part of:
  \begin{equation}
   \tau_\mu \nabla_\tau \sV^\nu + \tau^\nu \sd_\mu V_\tau
   \label{}
  \end{equation}
  and $BV$ is the anti-symmetric part of:
  \begin{equation}
   \tau_\mu(\nabla_\tau \sPi^{\nu'}_\nu) V_{\nu'}
   +\tau_\nu \sPi^{\nu'}_\mu(\nabla_{\mu'} \tau^{\nu'}) V_{\nu'}.
   \label{}
  \end{equation}
  To get \eqref{he:sdiv} and \eqref{he:scurl}, by \eqref{assump:rm},
  it suffices to control
  $\sd^{r-1} \nabla_\tau V$. To do this, we write $\nabla_\tau
  = \frac{1}{\Vtau}\bigg( \nabla_V - \Vbar^\mu \nabla_\mu\bigg)$ and we have:
  \begin{equation}
   |\sd^{r-1} \nabla_\tau V| \leq \frac{1}{|\utau|}
   \bigg( |\sd^{r-1} \nabla_u V| + |\ubar| |\sd^{r} V|
   + |\sd^{r-1} u| |\nabla V|\bigg)
   + P(|\nabla^{r-3}u, R) |\nabla^{r-1}V|.
   \label{}
  \end{equation}

  By \eqref{assump:stat}, this implies \eqref{he:sdiv} and
  \eqref{he:scurl}.

 \end{proof}

\section{Energy estimates for solutions of Euler's equations}
\label{energysec}

The energies for Euler's equation are:
\begin{multline}
 \E^{k,\ell}(t) = \frac{1}{2}\int_{\D_t} (\tau_\mu \tau_\nu + \sg_{\mu\nu})
 Q(\sd^k \nabla_u^\ell V^\mu, \sd^k\nabla_u^\ell V^\nu)\frac{\sqrt{\sigma}}{|\ubar|
 -\utau}\\
 +
 \frac{1}{4}\int_{\D_t} e'(\sigma)Q(\sd^k \nabla_u^\ell
 \sigma, \sd^k \nabla_u^\ell \sigma) (-\Vtau)\dvols\\
 + \frac{1}{4}\int_{\pa \D_t} Q(\sd^k \nabla_u^\ell \sigma, \sd^k \nabla_u^\ell \sigma)
 \frac{(-\Vtau)}{|\nabla_N \sigma|} \dvolbdy,
 \label{en:ekldef}
\end{multline}
\begin{equation}
 \K^{r}(t) = \int_{\D_t} | \curl\sd^{r-1} V|^2 \dvols
  \label{}
\end{equation}
and the energies for the wave equation \eqref{wave:eq}
satisfied by $\sigma$ are:
\begin{equation}
 \EW^r(t) = \int_{\D_t} |\nabla_{u}^{r+1} \sigma|^2
 + \aPi^{\mu\nu} (\nabla_\mu \nabla_u^r \sigma)
 (\nabla_\nu \nabla_u^r \sigma) \dvols
 + \int_{\D_t}  |\nabla_u^{r+1}\sigma|^2(\eta^{-2} -1) (-\utau) \dvols.
 \label{}
\end{equation}

We define:
\begin{equation}
 \E^r(t) = \sum_{k + \ell \leq r} \E^{k,\ell}(t)+
 \K^{r}(t)
 + \EW^r(t).
 \label{}
\end{equation}

\begin{theorem}
  \label{mainthm}
 Suppose that \eqref{assump:rm}, \eqref{assump:baro}, \eqref{assump:sound}
 and \eqref{assump:ecoer} hold for $N \geq r+1$ and that the
 a priori assumptions \eqref{assump:stat}, \eqref{assump:fluid1},
 \eqref{assump:fluid2} and \eqref{assump:tsc2} hold.
 For $r \geq 0$, there is a
 continuous function $C = C(r, M,\lambda, L, \delta^{-1},t,R)$ and
 a polynomial so that
 if $V, \sigma$ satisfy \eqref{int:dtv}-\eqref{int:dtesig}
 and \eqref{int:norm}-\eqref{int:free} for $0 \leq t \leq T$,
 then
 \begin{equation}
  \E^r(t) \leq \E^r(0) + C \int_0^t \E^r(s)
  + P(\E^{r-1}(s)) \,ds,
  \label{main}
 \end{equation}
 for $0 \leq t \leq T$.

 In particular, there are continuous functions
 $C_r = C_r(t,  M,\lambda, L, \delta^{-1}, R, \E^{r-1}(0))$ so that:
 \begin{equation}
  \E^r(t) \leq C_r \E^r(0).
  \label{gron}
 \end{equation}
\end{theorem}
Furthermore, for large enough $r$ we can get back control of the a
priori assumptions. Let:
\begin{align}
 K(t) &= ||\theta(t,\cdot)||_{L^\infty(\pa \D_t)} + \frac{1}{\iota_0(t)},\\
 \widetilde{\sigma}(t) &= ||(\sigma(t,\cdot)^{-1}||_{L^\infty(\D_t)},\\
 \delta(t) &= || (\sd p(t,\cdot)^{-1}||_{L^\infty(\pa \D_t)},\\
 \lambda(t) &= || \ubar(t) (\utau(t))^{-1}||_{L^\infty(\D_t)}.
 \label{}
\end{align}
\begin{theorem}
  \label{bootstrapthm}
  If $r \geq 5$, $\lambda(0) < \lambda^*/2$ where $\lambda^*$ is defined in
  \eqref{difflem}, and the assumptions \eqref{assump:rm},
  \eqref{assump:baro}, \eqref{assump:sound},
  \eqref{assump:ecoer} hold with $N \geq r+1$,
  then there is a continuous function:
  \begin{equation}
    \T_r = \T_r(K(0), \lambda(0),\delta(0),
  L,R, \E^r(0), Vol\D_t),
  \end{equation}
   so that for $0 \leq T \leq \T_r$:
  \begin{equation}
  \E^r(t) \leq 2 \E^r(0).
  \label{bootstrap}
 \end{equation}
\end{theorem}

As a first step, we show that the energies
\eqref{en:ekldef} control all derivatives of $V$ and $\sigma$:
\begin{lemma}
  \label{coerlem1}
 If $V, \sigma$ satisfy \eqref{int:dtv}-\eqref{int:free2} then there is a polynomial $P$
 so that for $k + \ell = r$:
 \begin{align}
  ||\nabla^r V||_{L^2(\D_t)}^2 +
  ||\nabla^{r} \sigma||_{L^2(\D_t)}^2 \leq
  C(\lambda) \E^{r}(t) + P(M,\lambda, L, R, \E^{r-1}(t)).
  \label{coerfull} \end{align}
 and
 \begin{align}
  ||\Pis \sd^{r}\sigma||_{L^2(\pa \D_t)}^2
  \leq ||\sd \sigma||_{L^\infty(\pa \D_t)}
  \E^{r}(t).\label{coerbdy}
\end{align}
\end{lemma}
\begin{proof}
 By \eqref{dtvf}, to prove \eqref{coerfull} it suffices to bound
 $||\sd^k \nabla_u^\ell V||_{L^2(\D_t)} +
 ||\sd^k \nabla_u^{\ell}\sigma||_{L^2(\D_t)}$ when
 $k + \ell = r$. By \eqref{sigmaellbd},
 $||\sd^k \nabla_u^\ell \sigma||_{L^2(\D_t)}$
 is bounded by the right-hand side of \eqref{coerfull}.

To control $\sd^k \nabla_u^\ell V$, when $\ell = 0$ we just use
\eqref{ellpw}, the estimates \eqref{he:sdiv},\eqref{he:scurl} and
the bounds for $\sigma$. For $\ell \geq 1$, we use that
$\nabla_u V = \frac{1}{2\sqrt{\sigma}} \nabla \sigma$ and the
bounds for $\sigma$.
 The estimate \eqref{coerbdy} just follows from the boundary term in
 \eqref{en:ekldef}.
\end{proof}

 To close the energies, we will need to show that we control derivatives
 of $\theta$. As in \cite{CL00}, the key point is that
 if the Taylor sign condition holds, then the boundary term in our
 energy actually gives us the control of $\theta$ that we need:
\begin{prop}
  \label{coerprop1}
  If the above assumptions hold:
  \begin{equation}
   ||\sd^{r-2} \theta||_{L^2(\pa \D_t)}^2 \leq
    C \E^r(t).
   \label{thetabd}
  \end{equation}
  where $C = C(M, K, R, ||\theta||_{L^\infty(\pa \D_t)},
    ||(\sd_\N \sigma)^{-1}||_{L^\infty(\pa \D_t)},
    Vol (\D_t), \E^{r-1}(t))$.
\end{prop}
\begin{proof}
 Because $|\sd_\N \sigma| \geq \delta >0$, this follows from \eqref{projtheta}.
\end{proof}

The last ingredient we will need is an estimate for
$\Pis \sd^r \nabla_V \sigma$ on the boundary:
\begin{lemma}
  Under the above hypotheses:
 \begin{equation}
  ||\Pis \sd^r \nabla_V \sigma(t)||_{L^2(\pa \D_t)}^2 +
  ||\sd^{r-1} \nabla_V \sigma(t)||_{L^2(\pa \D_t)}^2
  + ||\sd^r \nabla_V \sigma(t)||_{L^2(\D_t)}^2
  \leq C \E^{r}(t),
  \label{hotbdy}
 \end{equation}
 where $C = C(M, \lambda, L, K, R,\E^0(t),...,\E^{r-1}(t))$.
\end{lemma}
\begin{proof}
  Adding \eqref{projest} and \eqref{ellbdyfn2}
  and using \eqref{elllotbdy1}-\eqref{elllotbdy2}
  to deal with the lower order terms shows that for any $\delta > 0$, the
  left-hand side of \eqref{hotbdy} is bounded by:
  \begin{multline}
  C(R, ||\theta||_{L^\infty(\pa \D_t)}, ||\theta||_{H^{r-3}(\pa \D_t)})
  \bigg( \delta ||\Pis \sd^r \nabla_V \sigma||_{L^2(\pa \D_t)}
  + ||\sDelta \nabla_V \sigma||_{H^{r-2}(\D_t)}\\
   + ||\sd^{r-2} \theta||_{L^2(\pa \D_t)}
  ||\sd_\N \nabla_V \sigma||_{L^\infty(\pa \D_t)}\bigg).
   \label{}
  \end{multline}
  To bound $||\sDelta \nabla_V \sigma||_{H^{r-2}(\D_t)}$, we use
  the equation
  \eqref{ellform1} and \eqref{sigmaellbd}:
  \begin{equation}
   ||\sDelta \nabla_V \sigma||_{H^{r-2}(\D_t)}
   \leq C\bigg(\EW^r(t) + P(||\sigma||_{r}, ||V||_{r-1})\bigg),
   \label{}
  \end{equation}
   where $C = C(M,\lambda,L, K, Vol(\D_t), ||\theta||_{H^{r-2}(\pa \D_t)})$.
  Using \eqref{coerfull} and \eqref{thetabd} and taking
  $\delta$ sufficiently small, we have \eqref{hotbdy}.

\end{proof}

\begin{proof}[Proof of Theorem \ref{mainthm}]
  By theorem \ref{sigmawavebd} and the
  estimates \eqref{coerfull} and \eqref{coerbdy},
  $\EW^r(t)$ is bounded by the right-hand side of
  \eqref{main}.

  We first show how to control $\E^{r,0}$ and $\K^{r}$ and then
  we sketch the argument for $\E^{k,\ell}$ when $\ell \geq 1$.
  Using \eqref{he:mom} it is straightforward to prove:
  \begin{equation}
    \K^r(t) \leq \K^r(0) + C(M, R)\int_0^t ||\nabla^r V(s)||_{L^2(\D_t)}
    \sqrt{\K^r(s)}\, ds,
    \label{}
  \end{equation}
  and this is  bounded by the right-hand side of \eqref{main}.

  To bound $\E^{r,0}(t)$,
  we set $\beta = \sd^{r-1} V$,
  $\alpha = \frac{1}{2}\sd^{r-1}\sigma$, $\nu = -\sd \sigma$ and $\kappa
  = 2e'(\sigma)$, and apply \eqref{en:ident}:
  \begin{equation}
    \E^{r,0}(t) - \E^{r,0}(0)
    \leq C \int_0^t \sqrt{\E^{r,0}(s)}
    \bigg( \sqrt{\Rr_1(s)} + \sqrt{\Rr_2(s)} + \sqrt{\Rr_3(s)}\bigg)\, ds,
    \label{}
  \end{equation}
  where
  \begin{align}
    \Rr_1(t) &=
    ||\nabla_v\sd^r V + \tfrac{1}{2} \nabla \sd^r \sigma||_{L^2(\D_t)}^2
    + ||\div \sd^r V + e'(\sigma) \nabla_V \sd^r \sigma||_{L^2(\D_t)}^2,\\
    \Rr_2(t) &= ||\Pis \big(N^\mu \sd^r V_\mu -
    (\nabla \sigma)^{-1} \sd^r \sigma\big)||_{L^2(\pa \D_t)}^2,
  \end{align}
  and
  \begin{multline}
    \Rr_3(t) =
    (C_1 + C_2 + K)\E^{r,0} +
    ||\tfrac{1}{2} \sd^r \sigma + V^\mu \sd^r V_\mu||_{L^2(\D_t)}^2
    + ||\sDelta \sd^{r-1} \sigma||_{L^2(\D_t)}\\
    + ||\sdiv \sd^{r-1} V||_{L^2(\D_t)} + ||\scurl \sd^{r-1} V||_{L^2(\D_t)}
    \label{}
  \end{multline}
  By \eqref{he:mom}, \eqref{he:mass}, \eqref{he:sdiv},
  \eqref{he:scurl} and \eqref{coerfull}, $\sqrt{\Rr_1} + \sqrt{\Rr_3}$ is
  bounded by the right-hand side of \eqref{main}.

  By \eqref{he:bdy} and the estimate \eqref{hotbdy}, to control
  $\Rr_2(t)$ it is enough to bound:
  \begin{equation}
   ||\Pis \big( (\sd^s V^\mu)
   (\nabla_\mu \sd^{r-s}\sigma)\big)||_{L^2(\pa \D_t)},
   \label{}
  \end{equation}
  for $s = 1,..., r-1$.
  As in \cite{CL00}, each of these terms is clearly lower
  order and we could use the elliptic estimates \eqref{ellbdy1} to control
  each of these. However this would lead to estimates which are not linear
  to highest order. Instead, it is more efficient to use the interpolation
  estimate \eqref{bdyinterpu} which would apply except that one of the derivatives
  on $\sigma$ is not tangential. See the discussion after equation
  (7.21) in \cite{CL00} for how to deal with the case
  that the derivative $\nabla_\mu$ is purely spatial. For the component
  of $\nabla_\mu$ parallel to $\tau_\mu$, we use \eqref{matderivident}
  and it suffices to control:
  \begin{align}
   ||\Pis \big( (\sd^s V^\mu)(\sV^\nu \nabla_\nu
    \sd^{r-s}\sigma)\big)||_{L^2(\pa \D_t)},&&
   ||\Pis \big( (\sd^s V^\mu)(\nabla_V
    \sd^{r-s}\sigma)\big)||_{L^2(\pa \D_t)}.
   \label{}
  \end{align}
  The first term just involves spatial derivatives and so can
  be dealt with by interpolation as above, and the second term
  is lower order by \eqref{int:free2} and the estimate
  \eqref{projest}.

  To get control of $\E^{k,\ell}$ for $\ell \geq 1$, we note that
  because $\nabla_u^\ell \sigma = 0$
  for any $\ell \geq 1$, the estimate \eqref{projest} tells us that
  the boundary term in $\E^{k,\ell}$ is actually below top order.
  We also have
  $\sd^{k}\nabla_u^\ell V^\mu = -\frac{1}{2} \sd^k \nabla_u^{\ell-1}
  \bigg(\frac{1}{\sqrt{\sigma}} \nabla^\mu \sigma\bigg)$,
  so by \eqref{assump:lower}
  and \eqref{sigmaellbd},
  this can be bounded by the right-hand side of \eqref{main}.
\end{proof}
\begin{proof}[Proof of Theorem \ref{bootstrapthm}]
  The proof is essentially the same as the proof of Lemma 7.6
of \cite{CL00}. The point is that by the Sobolev inequalities
\eqref{bdysob2} and \eqref{intsob2} and the formula \eqref{projbigo},
 there are continuous functions $F_j$ so that if
$r_0 \geq 3/2 + 2$:
 \begin{align}
  ||\nabla u||_{L^\infty(\D_t)} +
  ||\nabla \sigma||_{L^\infty(\D_t)} &\leq F_1(K, \E_r, r),\label{Mbd}\\
  ||\nabla^2 \sigma||_{L^\infty(\pa \D_t)} &\leq F_2(K, \E_r, r),\label{Mbd2}\\
  ||\theta||_{L^\infty(\pa \D_t)} &\leq F_4(K, \E_r, r),\label{Mbd3}\\
  ||\nabla \nabla_V \sigma||_{L^\infty(\pa \D_t)} &\leq F_4(K, \E_r, r)
  \label{Mbd4}.
 \end{align}
 Arguing as in the proof of Lemmas 7.7-7.9 from \cite{CL00} and using
 the above estimates gives \eqref{bootstrap}.

\end{proof}

\section*{Acknowledgements}
The author would like to thank his advisor Hans Lindblad for many lengthy
and helpful discussions.

\appendix

  \section{Interpolation and Sobolev inequalities}
  The results of this section are well-known. However, what is important
  here is that the constants in the below inequalities can all be
  bounded in terms of the injectivity radius.
  The proofs of these theorems with these constants
  appear in the appendix to  \cite{CL00}
  when $Rm = 0$, and they carry over
  to the general case without change.
  \subsection{Interpolation inequalities}
  We will require interpolation inequalities both on $\pa \Omega$
  and $\Omega$.

  \begin{lemma}
    Suppose that:
   \begin{equation}
    \frac{m}{s} = \frac{k}{p} + \frac{m-k}{q},
    2 \leq p \leq s \leq q \leq \infty,
    \label{}
   \end{equation}
   and let $a = k/m$. Then there is a constant $C$ depending only
   on $m$ so that for any $(0,r)$ tensor $\alpha$:
   \begin{equation}
    ||\nas^k \alpha||_{L^s(\pa \Omega)} \leq C||\alpha||_{L^q(\pa \Omega)}^{1-a}
    ||\nas^m \alpha||_{L^p(\pa \Omega)}^a.
    \label{bdyinterp}
   \end{equation}
   In addition,  if $\iota_0 \geq \frac{1}{K}$, then:
  \begin{equation}
     \sum_{j = 0}^k ||\sd^j \alpha||_{L^s(\Omega)}
     \leq C||\alpha||_{L^q(\Omega)}^{1-a}\bigg(\sum_{i = 0}^m
     ||\sd^i \alpha||_{L^p(\Omega)} K^{m-i}\bigg)^a.
     \label{intinterp}
    \end{equation}
    In particular, if $\ell + m = k$ then:
    \begin{equation}
      ||\nas^\ell \alpha \nas^m \beta||_{L^2(\Omega)}
      \leq C\big(||\alpha||_{L^\infty(\pa \Omega)} ||\beta||_{H^k(\pa \Omega)}
      + ||\beta||_{L^\infty(\pa \Omega)} ||\alpha||_{H^k(\pa \Omega)}),
      \label{bdyinterpu}
    \end{equation}
    and
    \begin{equation}
      ||\sd^\ell \alpha \sd^m \beta||_{L^2(\Omega)}
      \leq C(K)\big(||\alpha||_{L^\infty(\Omega)} ||\beta||_{H^k(\Omega)}
      + ||\beta||_{L^\infty( \Omega)} ||\alpha||_{H^k(\Omega)}).
      \label{intinterpu}
    \end{equation}
  \end{lemma}
  \begin{proof}
   The inequalities \eqref{bdyinterp} and \eqref{intinterp}
   are (A.4) and (A.12) in \cite{CL00}. The inequalities
   \eqref{bdyinterpu} and \eqref{intinterpu} folow from \eqref{bdyinterp}
   and \eqref{intinterp} by H{\"o}lder's inequality; see eg. \cite{Taylor2011a}.
  \end{proof}

  \subsection{Sobolev and Poincar\'e inequalities}
  \begin{lemma}
   Suppose that $1/\iota_0 \leq K$. Then for any $(0,r)$-tensor:
   \begin{align}
     ||\alpha||_{L^{2p/2-kp}(\pa \Omega)}
     &\leq C(K) \sum_{\ell = 0}^k ||\sd^\ell\alpha||_{L^p(\pa \Omega)},
     && 1 \leq p \leq \frac{2}{k},\label{bdysob1}\\
     ||\alpha||_{L^\infty(\pa \Omega)}
     &\leq C(K)\sum_{0 \leq \ell \leq k-1}
     ||\sd^\ell \alpha||_{L^p(\pa \Omega)}, && k > \frac{2}{p},
    \label{bdysob2}
   \end{align}
   and
   \begin{align}
     ||\alpha||_{L^{3p/3-kp}(\Omega)}
     &\leq C(K) \sum_{\ell = 0}^k ||\sd^\ell\alpha||_{L^p(\Omega)},
     && 1 \leq p \leq \frac{3}{k},\label{intsob1}\\
     ||\alpha||_{L^\infty( \Omega)}
     &\leq C(K)\sum_{0 \leq \ell \leq k-1}
     ||\sd^\ell \alpha||_{L^p(\Omega)}, && k > \frac{3}{p}.
    \label{intsob2}
   \end{align}
  \end{lemma}
  \begin{remark}
    For \eqref{bdysob1} and \eqref{bdysob2} one can instead think of
    the norm on the left-hand side as being defined in terms of
    $\gamma^{IJ}\alpha_I\alpha_J$ and replace the derivatives
    $\sd$ with $\nas$  on the right-hand side.
  \end{remark}
  We will also need the following version of the Poincar\'e lemma, whose
  proof is in \cite{CL00}.
  \begin{lemma}
   If $q = 0$ on $\pa \Omega$ then:
   \begin{align}
    ||q||_{L^2(\Omega)} \leq C (Vol \Omega)^{1/3} ||\sd q||_{L^2(\Omega)},
    \label{poin}\\
    ||\sd q||_{L^2(\Omega)} \leq C (Vol \Omega)^{1/6}
    ||\sDelta q||_{L^2(\Omega)}\label{poin2}.
   \end{align}
  \end{lemma}
\section{Estimates for remainder terms}
  We begin by collecting the various assumptions we will need about the
  background spacetime $(M, g)$ and the equation of state:
  \begin{align}
   \sum_{s =1}^N |\nabla^s Rm| + |\nabla^s \tau| &\leq R,\label{arm}\\
   \bigg| \frac{d^k}{d\epsilon^k} p(\epsilon)\bigg|
   &\leq L_1, \textrm{ for } k \leq N,\label{abaro}\\
   0 < L_2 \leq \eta^2 &\leq 1,\label{asound}\\
   \bigg| \frac{d^k}{d\epsilon^k} p(\epsilon)\bigg|
   &\leq L_3, \textrm{ for } k \leq N,
   \label{abaro2}
  \end{align}
  for some large $N$, and we are writing $L = L_1 + (L_2)^{-1} + L_3$.
  We are also assuming the following bounds for $u$ and $\sigma$:
  \begin{align}
   |\nabla u| + |\nabla \sigma| \leq M,\label{aM}\\
   \lambda \equiv \frac{|\ubar|}{|\ubar|} \leq \lambda^*,
   \label{asmall}
  \end{align}
  for sufficiently small $\lambda^*$.

  \subsection{Estimates for $\nabla_\tau$}
    We need a few simple results to bound time derivatives
    $\nabla_\tau$ in terms of material derivatives $\nabla_u$ and
    $\ubar \sd$. By definition:
    \begin{equation}
     \nabla_\tau = \frac{1}{\utau} \big( \nabla_u - \ubar^\mu \nabla_\mu\big).
     \label{mat2}
    \end{equation}
    The point of this identity is that $\nabla_u$ is a tangential
    derivative and $\ubar$ is small by assumption \eqref{asmall}.

    We set:
    \begin{equation}
      |X|_s = \sum_{ k + \ell \leq s} |\sd^k \nabla_u^\ell X|.
    \label{}
    \end{equation}

    Then:
    \begin{lemma}
    \label{difflem}
    Suppose that
    \eqref{arm} holds with $N \geq r$. There is a $\lambda^* = \lambda^*(r)$
     so that if $u$
    satisfies \eqref{asmall} with $\lambda < \lambda^*$, there is a
    polynomial $P$ so that:
    \begin{equation}
      |\nabla^r X| \leq C(\lambda)\big( |X|_r +
      |u|_{r-1} |X|_1 + P( |u|_{r-2}, R) |X|_{r-1}\big),
       \label{dtvf}
    \end{equation}
    and for any function $\psi$, if $k + \ell = r-2$:
    \begin{equation}
     |\sd^k \nabla_u^\ell \nabla_\tau^2 \psi| \leq
     |\sd^k \nabla_u^{\ell+2}
      \psi|
      +\lambda^2 |\sd^{k+2}\nabla_u^\ell\psi|
      + C(\lambda)\big(|u|_{r-1}|\psi|_1
     + P(|u|_{r-2})|\psi|_{r-3}\big)
     \label{dtfn}
    \end{equation}
    \end{lemma}
    \begin{proof}
      For $r = 1$, \eqref{dtvf} follows directly
       from \eqref{mat2}. For $r \geq 2$, it suffices to bound
      $|\sd^k \nabla_\tau^\ell X|$ for $k + \ell = r$, and applying
      \eqref{mat2}, we have:
      \begin{multline}
        |\sd^k \nabla_\tau^\ell X|
        \leq \frac{1}{|\utau|^\ell} \bigg(
        |\sd^k \nabla_u^\ell X| + |\ubar|^\ell |\sd^{k+\ell} X|\bigg)
         + |\sd^k \nabla_\tau^{\ell-1} u||\nabla X|
         \\
          + P(|\nabla^{r-2}u|, |\utau|^{-1}, R) \sum_{s = 0}^{r-1}
        |\nabla^{s-1} X|. \label{dtvf2}
      \end{multline}
      The result then follows by
      induction if $\lambda$ is taken sufficiently small.

      To get \eqref{dtfn}, we use \eqref{mat2} twice:
      \begin{equation}
       |\sd^k \nabla_u^\ell \nabla_\tau^2 \psi|
       \leq \frac{1}{|\utau|^2} \bigg( |\sd^k \nabla_u^{\ell+2}\psi|
       + |\ubar|^2 |\sd^k \nabla_u^\ell \sd^2 \psi|\bigg)
       + |\sd^k\nabla_u^\ell \nabla u| |\nabla \psi|
       + P(|u|_{r-2},R) |\psi|_{r-1}.
       \label{}
      \end{equation}
      The factor $|\sd^k \nabla_u^\ell \nabla u|$ can be bounded using
      \eqref{dtvf}. We also have:
      \begin{equation}
       |\sd^k \nabla_u^\ell \sd^2 \psi|
       \leq |\sd^{k+2} \nabla_u^\ell \psi| +
       |\sd^k \nabla_u^{\ell-1} u| |\sd \psi| +
       P(|u|_{r-4},R) |\psi|_{r-2},
       \label{}
      \end{equation}
      which implies \eqref{dtfn}.
      \end{proof}
  \subsection{Estimates for $\F_r, e_r, g_r$}
  Recall that we have defined:
  \begin{align}
   \F_{k,\ell} &= -\sd^k \nabla_u^\ell \big( (\nabla_\mu V^\nu)(\nabla_\nu V^\mu)\big),
   \\
   \G_{k,\ell}(\sigma) &= -\sd^k\nabla_u^\ell\big(
   (\nabla_u e(\sigma) + \nabla_u  e'(\sigma)) \nabla_u \sigma\big),\\
   g_{k,\ell}^1(\sigma) &= \sd^k \big( \sDelta \nabla_u^\ell \sigma - \nabla_u^\ell
   \sDelta \sigma\big),\\
   g_{k,\ell}^2(\sigma) &= \sd^k\nabla_u^\ell \nabla_\tau^2 \sigma
   - \nabla_\tau^2 \sd^k \nabla_u^\ell \sigma,\\
   e_{k,\ell}(\sigma) &= \eta^{-2} \sd^k \nabla_{u}^{\ell+2}\sigma
   -\sd^k \nabla_u^\ell(\eta^{-2}\nabla_u^2 \sigma).
   \label{}
  \end{align}
  and that we are writing:
  \begin{equation}
   ||f||_r = \sum_{k + \ell \leq r} ||\sd^k \nabla_u^\ell f||_{L^2(\Omega)}.
   \label{}
  \end{equation}

  We first compute:
  \begin{equation}
   [\nabla_\mu, \nabla_u]f = -(\nabla_\mu u^\nu) \nabla_\nu f.
   \label{}
  \end{equation}
  To bound the above remainder terms,
  we will need a higher-order version of this, as well as estimates for
  derivatives of $e(\sigma)$:
  \begin{lemma}
   If $r \geq 5$,\eqref{arm}-\eqref{abaro2}
    hold for $N \geq r+1$, and \eqref{asmall} holds with sufficiently
    small $\lambda$, then for any $k + \ell = r$,
   there is a polynomial $P$ so that:
   \begin{equation}
    ||\sd^k \nabla_u^\ell f - \nabla_u^\ell \sd^k f
    + (\sd^k \nabla^{\ell-1} u^\mu)\nabla_\mu f||_{L^2(\Omega)}
    \leq P(\lambda, R, ||u||_{r-1}) \sum_{s = 2}^{r-1}
    ||\nabla^s f||_{L^2(\Omega)},
    \label{commest}
   \end{equation}
   and a polynomial $P$ so that:
   \begin{equation}
    ||\sd^k \nabla_u^\ell e'(\sigma) -
    e'(\sigma) \sd^k\nabla_u^\ell \sigma||_{L^2(\Omega)}
    \leq P(L, \lambda, R, ||\sigma||_{r-1}).
    \label{crest}
   \end{equation}
  \end{lemma}
  \begin{proof}
   We write $D^j = \sd^{m} \nabla^{j-m}$ for some $m$.
   Then, using \eqref{perm}, we need to estimate a sum of terms of the form:
   \begin{equation}
    u^m \cdot (D^{r_1}u)^{p_1} \cdots (D^{r_j} u)^{p_j}
    (D^{s} f),
    \label{}
   \end{equation}
   where $\sum_{i = 1}^j p_j + m = \ell$, $\sum_{i = 1}^j
   r_i p_i + s = r$ and $s \leq r-1, r_i \leq r-2$. If
   the $r_i$ and $s$ are all less than $r-3$ then we get the result
   by Sobolev embedding and \eqref{dtvf}. When $s = r-2$, then there can't be any
   $r_i$ larger than 2 and since $r \geq 5$ we can again apply Sobolev
   embedding and \eqref{dtvf}. Dealing with the other cases is similar.

   To control \eqref{crest}, we need to bound a sum of terms of the
   form:
   \begin{equation}
    e^{(m)}(\sigma) (D\sigma)^{p_1}\cdots (D^{j}\sigma)^{p_j},
    \label{}
   \end{equation}
   where $m \leq r + 1$, $p_1 + 2p_2 + ... + jp_j = r$ and $j \leq r-1$.
   If $j \leq r-3$ then we can again use Sobolev embedding to get the
   result. If $j = r-1$, then note that if $r\geq 3$, this forces $p_{r-1}
   = 1$. Then we have:
   \begin{equation}
    \sum_{\ell = 1}^{r-2} \ell p_\ell + r-1 = r,
    \label{}
   \end{equation}
   and so the only option is that $p_1 = 1, p_{\ell}  = 0$ otherwise, and
   the result is then clearly bounded by the above. The case $j = r-2$ is similar.
   \end{proof}

   This then implies:
  \begin{lemma}
    \label{remainderlem}
    Let $r \geq 5$. Suppose that the assumptions of the previous lemma hold.
    For $k + \ell = r$:

\begin{align}
 ||\F_{k,\ell}||_{L^2(\Omega)}
 &\leq C_1 \bigg( ||\nabla \sd^k \nabla_u^\ell V||_{L^2(\Omega)} + P(||V||_r,
 ||\sigma||_r, R)\bigg),\label{fest}\\
 ||g_{k,\ell}^1||_{L^2(\Omega)} + ||g_{k,\ell}^2||_{L^2(\Omega)}
 + ||\widetilde{g}_r||_{L^2(\Omega)}
 &\leq C_1 \bigg(
 ||\nabla^{k+ \ell+1} u||_{L^2(\Omega)} + P(||u||_r)
 ||\sigma||_r\bigg),\label{gest}\\
 ||e_{k,\ell}||_{L^2(\Omega)}
 + ||\G_{k,\ell}||_{L^2(\Omega)} &\leq C_2 \bigg(
 ||\sd^k \nabla_u^\ell \sigma||_{L^2(\Omega)} + P(||u||_{r-1}, ||\sigma||_{r-1})
 \bigg),
 \label{eest}
\end{align}
where $C_1 = C_1(M, \lambda)$ and $C_2 = C_2(M, \lambda, L)$.
 \end{lemma}
  \begin{proof}
    The estimate \eqref{fest} follows from \eqref{commest}
    and the estimate \eqref{gest} follows from a straightforward
    modification of the proof of \eqref{commest}.

    To prove \eqref{eest}, we note that:
    \begin{equation}
     \eta^{-2} \sd^k \nabla_u^{\ell+2} \sigma
    - \sd^k \nabla_u^\ell(\eta^{-2} \nabla_u^2 \sigma)
    = C \eta^{-2-k-\ell} \eta^{(k+\ell)} \sd^k \nabla_u^\ell
    \sigma + R(\eta, \sigma),
     \label{}
    \end{equation}
    where we are writing $\eta^{(s)} = \frac{d^s}{d^s\sigma}\eta(\sigma)$
    and $R(\eta, \sigma)$ is a sum of terms of the form:
    \begin{equation}
     F(\eta, \eta', ..., \eta^{k+\ell-1}) (D^{r_1} \sigma)^{p_1}
     \cdots (D^{r_j} \sigma)^{p_r},
     \label{}
    \end{equation}
    where $F$ is bounded if the lower bound in\eqref{asound} holds,
    and where $\sum_{i = 1}^j
    r_i p_i = k+\ell$, $r_i \leq r-1$. As in the previous lemma,
    all these terms can be bounded by the right-hand side of \eqref{eest}
    by Sobolev embedding and \eqref{dtvf}. The estimate for $\G_{k,\ell}$
    is similar.
    \end{proof}


\end{document}